\documentclass[11pt,x11names,dvipsnames,svgnames]{amsart}
\usepackage{amsfonts,amssymb,amscd,amstext,mathrsfs,mathtools}
\usepackage{hyperref}
\usepackage{verbatim}
\usepackage[T1]{fontenc} 

\usepackage{tikz}
\usetikzlibrary{calc, arrows.meta, decorations.markings, shapes.geometric}

\usepackage{graphicx}
\usepackage[up,bf]{caption} 
\usepackage{subcaption}

\usepackage{times}
\usepackage{enumerate}
\input xy
\xyoption{all}

\numberwithin{equation}{section}
\numberwithin{figure}{section}

\newtheorem{theorem}{Theorem}[section] 
\newtheorem{proposition}[theorem]{Proposition} 
\newtheorem{claim}[theorem]{Claim} 
\newtheorem{lemma}[theorem]{Lemma} 
\newtheorem{corollary}[theorem]{Corollary} 

\theoremstyle{definition} 
\newtheorem{definition}[theorem]{Definition} 
\newtheorem{remark}[theorem]{Remark}

\newcommand\Acal{\mathcal{A}} 
 
\newcommand\Ccal{\mathcal{C}} 
\newcommand\Dcal{\mathcal{D}} 
 
\newcommand\Fcal{\mathcal{F}}

\newcommand\Pcal{\mathcal{P}}

\newcommand\Cscr{\mathscr{C}} 
\newcommand\Dscr{\mathscr{D}}

\newcommand\Jscr{\mathscr{J}}

\newcommand\B{\mathbb{B}} 
\newcommand\C{\mathbb{C}} 
\newcommand\D{\overline{\mathbb D}} 
\newcommand\CP{\mathbb{CP}} 
\renewcommand\D{\mathbb D}

\newcommand\N{\mathbb{N}} 
 
\newcommand\R{\mathbb{R}}

\newcommand\Z{\mathbb{Z}}

\newcommand\igot{\mathfrak{i}}

\renewcommand\igot{\mathfrak{i}}

\renewcommand\imath{\igot}

\newcommand\hra{\hookrightarrow}

\newcommand\wt{\widetilde} 
 
\newcommand\di{\partial} 
\newcommand\dibar{\overline\partial}

\newcommand\Id{\mathrm{Id}}

\def\Ell1{\mathrm{Ell_1}} 
\def\CEll1{\mathrm{CEll_1}}

\def\Jst{J_{\mathrm{st}}} 

\newcommand\Flux{\mathrm{Flux}}

\newcommand\Co{\mathrm{Co}}

\newcommand\Hom{\mathrm{Hom}}

\newcommand\boldA{\mathbf A}

\makeatletter
\@namedef{subjclassname@2020}{\textup{2020} Mathematics Subject Classification}
\makeatother

\begin{document}

\title{Families of proper minimal surfaces}  

\author{Antonio Alarc\'on and Franc Forstneri{\v c}}

\address{Antonio Alarc\'on, Departamento de Geometr\'{\i}a y Topolog\'{\i}a e Instituto de Matem\'aticas (IMAG), Universidad de Granada, Campus de Fuentenueva s/n, E--18071 Granada, Spain.}

\email{alarcon@ugr.es}

\address{Franc Forstneri{\v c}, Faculty of Mathematics and Physics, University of Ljubljana, Jadranska 19, 1000 Ljubljana, Slovenia}

\address{Franc Forstneri{\v c}, Institute of Mathematics, Physics, and Mechanics, Jadranska 19, 1000 Ljubljana, Slovenia}

\email{franc.forstneric@fmf.uni-lj.si}



\date{19 May 2026}


\begin{abstract}
Assume that $X$ is a connected, open, oriented smooth surface,
$B$ is a compact Euclidean neighbourhood retract, 
and $\Jscr=\{J_b\}_{b\in B}$ is a continuous 
family of complex structures on $X$ of local H\"older class 
$\Cscr^\alpha$ for some $0<\alpha<1$.
We construct a continuous family of $J_b$-conformal minimal 
immersions $u_b:X\to \R^3$, $b\in B$, properly projecting to $\R^2$
and having an arbitrary given family 
of flux homomorphisms $\Flux_{u_b}:H_1(X,\Z)\to\R^3$.
In particular, there are continuous families of proper $J_b$-holomorphic 
null immersions $X\to \C^3$ and of proper $J_b$-holomorphic 
immersions $X\to\C^2$, $b\in B$. 
\end{abstract}

\maketitle

%
%
\section{Introduction}\label{sec:intro}

In this paper, we establish the existence of families of proper minimal surfaces in Euclidean space $\R^3$ with varying but prescribed complex structures, depending continuously on a parameter in a compact Euclidean neighbourhood retract. 

The global theory of minimal surfaces in $\R^3$ has been a major focus of interest since the fundamental developments by Osserman in the 1960s \cite{Osserman1986}, with emphasis on the conformal properties of such surfaces under global assumptions such as completeness or properness; see e.g.\ \cite{MeeksPerez2004SDG}. Until not too long ago, it was widely believed that properness of a minimal surface in $\R^3$ imposes strong restrictions on its underlying complex structure. In this direction, it was conjectured that a proper minimal surface in $\R^3$ having either finite topology (Sullivan) or a proper projection into a plane (Schoen-Yau) must have parabolic conformal type in the sense of Ahlfors-Nevanlinna \cite{Nevanlinna1941,Ahlfors1947} (that is, it does not carry any nonconstant negative subharmonic function \cite[Section IV.6]{AhlforsSario1960}); see \cite[p.\ 18]{SchoenYau1997}  and \cite[Section 3.10]{AlarconForstnericLopez2021}. 
A counterexample to Sullivan's conjecture was 
given in 2003 by Morales \cite{Morales2003GAFA}, who constructed a proper minimal surface with the conformal type of the disc. Later, in 2012, Alarc\'on and L\'opez proved that every open Riemann surface is the complex structure of a minimal surface in $\R^3$ properly projecting into a plane \cite{AlarconLopez2012JDG}, thereby settling both conjectures in the most general possible way. 
Nevertheless, almost all minimal surfaces in $\R^3$ (in the natural topological sense of Baire category) are nonproper \cite[Corollary 1.5]{AlarconLopez2026PE}, so it is fair to say that proper ones 
are hard to find.

In this paper, $X$ is a connected, open, 
orientable smooth surface with a countable base of topology, $B$ is a compact subset 
of a Euclidean space $\R^m$ which admits a retraction from an open neighbourhood 
$U\subset \R^m$ (every finite CW complex is such; see 
\cite[Definition 1.5]{Forstneric2024Runge} and the subsequent discussion),  
and $\Jscr=\{J_b\}_{b\in B}$ is a continuous family of complex structures 
on $X$ of local H\"older class $\Cscr^\alpha$ for some 
$0<\alpha<1$. See Section \ref{sec:prelim} for the details. 
Here is our main result; see also Remark \ref{rem:hp}. 
This gives an affirmative answer to 
\cite[Problem 8.7 (c)]{Forstneric2024Runge}.

%
%
\begin{theorem}\label{th:CMI}
Assume that $X$, $B$, and $\Jscr$ are as above.
Then there exists a continuous map $u=(u_1,u_2,u_3):B\times X\to \R^3$ 
such that the map $u_b=u(b,\cdotp)=(u_{b,1},u_{b,2},u_{b,3}):X\to\R^3$
is a proper $J_b$-conformal minimal immersion for every $b\in B$.
Furthermore, $u$ can be chosen such that every map $(u_{b,1},u_{b,2}):X\to\R^2$ is 
proper and the flux $\Flux_{u_b}$ of $u_b$ equals any given continuous family 
of homomorphisms $\Fcal_b:H_1(X,\Z)\to\R^3$ for $b\in B$.
\end{theorem}

Let us clarify the technical notions in the statement of the theorem. Assume that $J$ is a complex structure on $X$, so $(X,J)$ is an 
open Riemann surface. 
Let $\Jst$ denote the standard complex structure on $\C$, given by
multiplication by $\imath=\sqrt{-1}$. 
The exterior differential $d$ on $X$ then splits 
in the sum $d=\di +\dibar$ of the $(J,\Jst)$-linear part 
$\di=\di_J$ and the $(J,\Jst)$-antilinear 
part $\dibar=\dibar_J$, 
with $J$-holomorphic functions being the kernel of $\dibar$. 
The operator 
$d^c = d^c_J= \imath\, (\dibar-\di) = 2\, \Im (\di)$ 
is the conjugate differential, 
and $dd^c = 2\imath\, \di\dibar = -2\imath \dibar \di$ 
is the Laplace operator whose kernel are the $J$-harmonic functions.
A smooth immersion $u=(u_1,u_2,\ldots,u_n):X\to \R^n$, $n\ge 2$, 
is $J$-conformal if and only if the Riemannian metric on $X$ obtained by pulling back the Euclidean metric on $\R^n$ via the immersion $u$, 
together with the chosen orientation on $X$, 
induces the given complex structure $J$; 
equivalently, if $\di u=(\di u_1,\di u_2,\ldots,\di u_n)$ 
satisfies the nullity condition
\[ 
	(\di u_1)^2+(\di u_2)^2+\cdots+(\di u_n)^2=0.
\] 
Assume now that $n\ge 3$. A $J$-conformal immersion 
$u:X\to \R^n$ is minimal, in the sense that it  
parameterises a minimal surface in $\R^n$, if and only if it is 
$J$-harmonic; equivalently, $\di u$ is a $J$-holomorphic 1-form. 
The {\em flux} of such a map $u$ is the homomorphism 
$\Flux_u:H_1(X,\Z)\to\R^n$ defined on any closed curve 
$C\subset X$ by $\Flux_u(C)=\oint_C d^c u$, 
where $d^c u$ is the conjugate differential 
(see \cite[Definition 2.3.2]{AlarconForstnericLopez2021}). 
Since $dd^c u=0$, $\oint_C d^c u$ only depends on the 
homology class $[C]\in H_1(X,\Z)$. 
(Recall that $H_1(X,\Z)\cong\Z^\ell$ for some $\ell\in \{0,1,\ldots,\infty\}$.)
In particular, $\Flux_u=0$ if and only if $u$ has a 
$J$-harmonic  conjugate $v:X\to\R^n$, and in this case, 
$F=u+\imath v:X\to\C^n$ is a $J$-holomorphic immersion whose 
differential $dF=(dF_1,dF_2,\ldots,dF_n)$ satisfies 
$(dF_1)^2 + (dF_2)^2 + \cdots + (dF_n)^2 =0$.
Such $F$ is called a $J$-holomorphic null curve in $\C^n$.

Let $\Jscr=\{J_b\}_{b\in B}$ be a family of complex structures on $X$
as above. A continuous map $F:B\times X\to \C^n$ is said to be 
{\em $\Jscr$-holomorphic} if the map 
$F_b=F(b,\cdotp):X\to\C^n$ is $J_b$-holomorphic for every $b\in B$. 
Taking $\Fcal_b=0$ for all $b\in B$ in Theorem \ref{th:CMI}
gives the following corollary, 
which provides an affirmative answer to 
\cite[Problem 8.7 (b)]{Forstneric2024Runge} for families
of holomorphic null immersions.

%
%
\begin{corollary}\label{cor:null}
If $X$, $B$, and $\Jscr$ are as in Theorem \ref{th:CMI},
then there exists a $\Jscr$-holomorphic map 
$F=(F_1,F_2,F_3):B\times X\to \C^3$
such that for every $b\in B$, the map 
$F_b=F(b,\cdotp)=(F_{b,1},F_{b,2},F_{b,3}):X\to\C^3$
is a $J_b$-holomorphic null immersion and 
the map $(\Re F_{b,1},\Re F_{b,2}):X\to\R^2$ is proper.
In particular, the map $F_b:X\to\C^3$ is proper for every $b\in B$. 
\end{corollary}

For $B$ a singleton, Theorem \ref{th:CMI} and Corollary \ref{cor:null}
were proved by Alarc\'on and L\'opez 
\cite{AlarconLopez2012JDG} in a more precise form 
with approximation on compact Runge sets in $X$. 
The higher dimensional case of minimal surfaces in $\R^n$ 
and null curves in $\C^n$ for $n>3$ was treated in 
\cite{AlarconForstneric2014IM,AlarconForstnericLopez2016MZ}
and \cite[Sec.\ 3.10]{AlarconForstnericLopez2021}. 
The h-principle for families of such maps 
(not necessarily proper) from a single open Riemann surface 
was obtained in \cite{ForstnericLarusson2019CAG}, 
while the parametric h-principle for the inclusion of the subset of 
complete immersions was established in \cite{AlarconLarusson2025JGA}.
Families of not necessarily proper or complete minimal surfaces and  
null curves from a variable family of open Riemann surfaces 
were first constructed by the second named author in 
\cite[Sec.\ 8]{Forstneric2024Runge}. The nontrivial addition 
in Theorem \ref{th:CMI} and Corollary \ref{cor:null} 
is to ensure properness of maps in such families. 

Taking the first two components of a map in Corollary 
\ref{cor:null} gives the following corollary,
which provides an affirmative answer to 
\cite[Problem 8.7 (a)]{Forstneric2024Runge}.

\begin{corollary}\label{cor:C2}
If $X$, $B$, and $\Jscr$ are as in Theorem \ref{th:CMI},
then there exists a $\Jscr$-holomorphic map 
$F=(F_1,F_2):B\times X\to \C^2$
such that for every $b\in B$, the $J_b$-holomorphic map
$F_b:X\to\C^2$ is an immersion and the $J_b$-harmonic map 
$\Re F_b=(\Re F_{b,1},\Re F_{b,2}):X\to\R^2$ is proper. 
Hence, the immersion $F_b$ is proper for every $b\in B$. 
\end{corollary}

%
%
The special case of Corollary \ref{cor:C2} for a single map is related to 
the more ambitious conjecture by 
Schoen and Yau \cite[p.\ 18]{SchoenYau1997}  
that {\em no hyperbolic open Riemann surface admits a proper 
harmonic map into $\R^2$}. 
The first counterexample to their conjecture 
was given by Bo{\v z}in \cite{Bozin1999IMRN}, 
who found an explicit example of a 
proper harmonic map $\D\to\R^2$ from the disc. 
Another counterexample was given by Forstneri{\v c} and
Globevnik \cite{ForstnericGlobevnik2001MRL}, 
who constructed a proper holomorphic map $(f_1,f_2):\D\to\C^2$ 
with nowhere vanishing components, so 
$(\log|f_1|,\log|f_2|): \D\to\R^2$ is a proper harmonic map.
Counterexamples with arbitrary finite topology were obtained by Alarc\'on and G\'alvez \cite{AlarconGalvez2011}.
Finally, a construction of a proper harmonic map from any open Riemann surface into $\R^2$, different from the one in \cite{AlarconLopez2012JDG}, was carried out by Andrist and Wold \cite{AndristWold2014}.
See also the discussion in \cite[Sec.\ 3.10]{AlarconForstnericLopez2021}.

Recently, Drinovec Drnov\v sek and Kali\v snik 
\cite{DrinovecKalisnik2026JGA} proved that for every
$X$ and $\Jscr=\{J_b\}_{b\in B}$ as above there exists a
$\Jscr$-holomorphic map $F:B\times X\to\C^2$ such that 
the real part $\Re F_b:X\to\R^2$ of $F_b$ is a proper 
$J_b$-harmonic map for every $b\in B$.
(Their result holds for all metric parameter spaces $B$.) 
For parameter spaces $B$ used in this paper,
Corollary \ref{cor:C2} improves their result in 
that the map $F_b:X\to \C^2$ is a proper {\em immersion}
for every $b\in B$. 
Nevertheless, assuming that $B$ is as in 
\cite[Theorem 1.6]{Forstneric2024Runge} and using the existence of 
continuous families of $J_b$-holomorphic immersions $X\to\C$ 
\cite[Corollary 8.3]{Forstneric2024Runge}, 
it is possible to upgrade the proof in \cite{DrinovecKalisnik2026JGA} 
to obtain a continuous family of proper $J_b$-holomorphic maps 
$F_b=(F_{b,1},F_{b,2}):X\to \C^2$ such that every 
component function $F_{b,i}:X\to\C$ for $b\in B$ and $i=1,2$ 
is an immersion. 

The restriction to compact parameter spaces $B$ in this paper is mainly
out of convenience since it enables a geometrically simpler construction
in the proofs. We believe that, with more work
and using some ideas and technical arguments from the construction in \cite{DrinovecKalisnik2026JGA}, the results 
could be extended to countably compact spaces of the type used in 
\cite[Theorem 1.6]{Forstneric2024Runge}, including noncompact ones. 
Furthermore, we expect that one can prove a Runge approximation theorem  
for families of proper conformal minimal immersions in $\R^n$
for any $n\ge 3$. 

Note that the families of open Riemann surfaces $\{(X,J_b)\}_{b\in B}$
in the above results are considerably more general than the 
Teichm\"uller families, in which the individual surfaces are 
quasiconformally equivalent to one another.
Let us look at the topologically simplest examples when the surface 
$X$ is either simply or doubly connected.
Every complex structure on $X=\R^2$ is biholomorphically 
equivalent to the standard complex structure $\Jst$ on $\C$ 
or to its restriction to the unit disc $\D\subset \C$. 
We can put both examples in a compact connected 
family parameterised by $B=[0,1]$.
Every complex structure on $X=\R^2\setminus \{0\}$ is 
equivalent to $\Jst$ restricted to $\C^*=\C\setminus \{0\}$   
or to an annulus $A_r=\{z\in\C:r<|z|<1\}$ for $0\le r<1$.
For any $0<r_0<1$ one can realise all these structures with 
$0\le r\le r_0$ in a smooth family with the parameter space $B=[0,1]$.

Our proof of Theorem \ref{th:CMI} strongly depends on 
the recent work \cite{Forstneric2024Runge} by the second named
author. For suitable parameter spaces $B$ (more precisely, 
for local Euclidean neighbourhood retracts, a class of topological
spaces which includes all finite CW complexes and more generally
all countable locally compact CW-complexes of finite dimension), 
the main result of that paper is an Oka principle with approximation for 
$\Jscr$-holomorphic maps $B\times X\to Y$
to any Oka manifold $Y$; see \cite[Theorem 1.6]{Forstneric2024Runge}.
Applications include the construction 
of continuous families of holomorphic immersions $X\to\C^n$
for any $n\ge 1$, of holomorphic null immersions $X\to\C^n$
for any $n\ge 3$, and of conformal minimal immersions 
$X\to\R^n$ with any given family of flux homomorphisms 
for $n\ge 3$; see \cite[Sec.\ 8]{Forstneric2024Runge}. 
The main novelty of the results in this paper is that we construct
families of proper maps of the given types; in the case of
holomorphic immersions $X\to\C^n$, this requires $n\ge 2$. 
Properness is a nontrivial condition which is not easly achieved even
for single maps. In fact, proper maps form a meagre (topologically small) 
subset in any class of maps considered above, so it is not surprising 
that it is difficult to construct families of proper maps. 
With the approximation results from \cite{Forstneric2024Runge} in hand, our proof broadly follows the original construction in \cite{AlarconLopez2012JDG} (see also \cite[Section 3.11]{AlarconForstnericLopez2021}) of a minimal surface in $\R^3$ with given complex structure and proper projection into a plane. However, our approach introduces several major differences required to adapt the argument to the parametric setting. In particular, to compensate the lack of a parametric version of the Runge approximation theorem for conformal minimal immersions with fixed component functions (see \cite[Corollary 4.8 and Theorem 4.9]{AlarconLopez2012JDG} or \cite[Section 3.7]{AlarconForstnericLopez2021}), which has been a crucial tool in all previous constructions of proper minimal surfaces with arbitrary complex structure, we employ a parametric version of the L\'opez-Ros deformation for minimal surfaces \cite{LopezRos1991JDG}, developed by Alarc\'on and L\'arusson in \cite{AlarconLarusson2025JGA}, together with improved parametric versions of Gromov's convex integration lemma \cite{Gromov1973} (see Section \ref{ss:CI}).

We mention several open problems related to the results of this paper.

The analogous results for single maps in the above mentioned works hold with approximation on 
compact Runge subsets of $X$. We expect that this generalisation is also possible in the present setting, but the proofs 
become more complicated from the technical viewpoint. 

The second question is whether an h-principle holds 
in these results. It has recently been shown by 
Vrhovnik \cite{Vrhovnik2026JGEA}
that every nonflat conformal minimal immersion $X\to\R^n$, $n\ge 3$,
is homotopic through a family of such immersions to a proper one.
Does the analogous result hold for families?

The maps in our main results are immersions. 
Every open Riemann surface $X$ admits 
many proper conformal minimal {\em embeddings}
$X\hra \R^n$ for any $n\ge 5$, and proper holomorphic 
null embeddings $X\hra \C^n$ for any $n\ge 3$ 
(see \cite{AlarconForstneric2014IM,AlarconForstnericLopez2016MZ},
and \cite[Sec.\ 3.10]{AlarconForstnericLopez2021}). 
What can be said in this respect for families? 

Assuming that the parameter space 
$B$ is a smooth manifold, one may expect that refinements
of the techniques used in this paper, which are available in 
\cite{Forstneric2024Runge}, yield families of immersions 
of the given types depending smoothly on the parameter. 
In such a case, one may also hope that the use of transversality methods would yield families 
of embeddings of the given types when the target Euclidean 
space has sufficiently big dimension. 
We also expect that the analogue of Theorem \ref{th:CMI} 
holds for continuous families of 
conformal structures on a nonorientable smooth open 
surface $X$. Results in this direction for a single conformal structure
on $X$ can be found in 
\cite{AlarconLopez2015GT,AlarconForstnericLopezMAMS}.

The final problem concerns the universal
family $V(g,n)$ of $n$-punctured compact Riemann 
surfaces of genus $g$ with $n\ge 1$. The corresponding 
parameter space is the Teichm\"uller space $T(g,n)$,
which can be holomorphically 
realised as a bounded contractible Stein domain in a 
complex Euclidean space, and the natural Teichm\"uller projection 
$\pi:V(g,n)\to T(g,n)$ is a holomorphic submersion 
whose fibre over any point $b\in T(g,n)$ is the 
$n$-puncture compact Riemann surface $(X,J_b)$ 
of genus $g$ with the complex structure determined by $b$.
It has recently been proved by the second named author that 
$V(g,n)$ is a Stein manifold \cite[Theorem 1.1]{Forstneric2025Universal}.
Furthermore, there is a holomorphic map $F:V(g,n)\to\C^N$
for some $N\in\N$ whose restriction $F_b=F(b,\cdotp):(X,J_b)\to\C^N$ 
to any fibre is a proper algebraic embedding 
\cite[Corollary 3.5]{Forstneric2025Universal}.
Can one choose $F$ such that $F_b$ is a proper algebraic  
null embedding for every $b\in T(g,n)$? 
If so then the real part $\Re F:V(g,n)\to\R^N$ 
restricted to any fibre is a complete proper minimal surface
with finite total curvature. 

%
%
%
%
\section{Preliminaries} \label{sec:prelim}

\subsection{Families of complex structures on surfaces}
\label{ss:complexstructures}
Let $X$ be a smooth, connected, orientable surface
endowed with a smooth Riemannian metric.
A complex structure on $X$ is given by an endomorphism $J:TX\to TX$ 
of its tangent bundle satisfying $J^2=-\Id$. Thus, $J$ is a section of the smooth vector bundle $T^*X\otimes TX\to X$ whose fibre over 
$x\in X$ is the space $\Hom(T_xX,T_xX)$ of linear maps 
$T_x X\mapsto T_xX$. (Note that every orientable vector bundle 
on an open surface $X$ is trivial. When $X$ is orientable, 
this holds in particular for the tangent bundle $TX$ and its 
derived vector bundles.) 
The multiplication by $\imath=\sqrt{-1}$ defines the standard
complex structure $J_\mathrm{st}$ on $\C$. 
A differentiable function $f:X\to\C$ is 
{\em $J$-holomorphic} if it satisfies the Cauchy--Riemann equation 
\[
	df_x \circ J_x = \imath\, df_x = J_\mathrm{st} df_x,
	\quad x\in X. 
\]
A complex structure $J$ is said to be of local H\"older class 
$\Cscr^{(k,\alpha)}$ for some $k\in\Z_+=\{0,1,2,\ldots\}$ and $0<\alpha<1$
if, for every relatively compact domain $\Omega\Subset X$, the restriction
$J|_\Omega\in \Gamma^{(k,\alpha)}(\Omega,T^*\Omega\otimes T\Omega)$
is a section of class $\Cscr^{(k,\alpha)}(\Omega)$ 
of the restricted vector bundle $T^*\Omega\otimes T\Omega\to \Omega$. 
(The H\"older norms are computed with respect to 
the given Riemannian metric on $X$; see 
Gilbarg and Trudinger \cite[Sect. 4.1]{GilbargTrudinger1983}.) 
For such $J$, there is an atlas $\{(U_i,\phi_i)\}_i$ of open sets 
$U_i\subset X$ covering $X$ and $J$-holomorphic charts 
$\phi_i :U_i \to \phi_i(U_i)\subset \C$ 
of class $\Cscr^{(k+1,\alpha)}(U_i)$. (See 
\cite[Theorem 5.3.4]{AstalaIwaniecMartin2009} among other references.) 
It follows that the smooth structure on $X$ determined by 
a complex structure $J$ of local class $\Cscr^{(k,\alpha)}$ is 
$\Cscr^{(k+1,\alpha)}$ compatible with the given smooth structure
on $X$; see \cite[Theorem 2.1]{BojarskiAll2005}.

A family $\Jscr=\{J_b\}_{b\in B}$ of complex structures on $X$
is said to be of class $\Cscr^{0,(k,\alpha)}$ if 
for any relatively compact domain $\Omega\Subset X$ the map 
$B\ni b\mapsto J_b|_{\Omega}\in \Gamma^{(k,\alpha)}(\Omega,T^*\Omega\otimes T\Omega)$ is continuous.
If $k=0$, we write $\Cscr^{(0,\alpha)}(\Omega)=\Cscr^\alpha(\Omega)$. 
The proofs of our main results immediately extend to the case
when $\Jscr$ is of class $\Cscr^{0,(k,\alpha)}$ for any $k\in\Z_+$,
and they yields maps as in Theorem \ref{th:CMI} and Corollary 
\ref{cor:null} which are of class $\Cscr^{0,(k+1,\alpha)}$ 
in the reference complex structure on $X$.
(See \cite[Theorem 8.2 and Corollary 8.6]{Forstneric2024Runge}
for the construction of families of not necessarily proper maps
of these classes.)

If the parameter space $B$ is a manifold of class $\Cscr^l$
for some $l>0$, we can also introduce the notion of a family $\Jscr$ 
of class $\Cscr^{l,(k,\alpha)}$; see \cite{Forstneric2024Runge}.
In view of \cite[Theorem 1.6]{Forstneric2024Runge}, 
one may expect that our main results can be extended
to families of conformal minimal immersions and null immersions
of class $\Cscr^{l,(k+1,\alpha)}$ whenever $0\le l\le k+1$
and $0<\alpha<1$. However, for simplicity we shall only consider 
the case $l=0$ in this paper, that is, 
continuous dependence on the parameter.

Fix a reference smooth complex structure $J$ on $X$.
By a theorem of Gunning and Narasimhan \cite{GunningNarasimhan1967},
the open Riemann surface $(X,J)$ admits a holomorphic immersion 
$z:X\to \C$, which therefore provides a local $J$-holomorphic coordinate 
on $X$ at every point. A family $\Jscr=\{J_b\}_{b\in B}$ 
of complex structures in the same orientation class as $J$ 
can then equivalently be given by 
a family of maps $\mu_b:X\to \D=\{\zeta\in\C:|\zeta|<1\}$
of the same local smoothness class as $J_b$
and the same regularity in the parameter $b\in B$,
with $\mu_b=0$ corresponding to the reference structure $J_b=J$.  
(See \cite[Sec.\ 2]{Forstneric2024Runge} for the details.)  
Such a function $\mu_b$ is called a Beltrami multiplier.
Given an open subset $U\subset X$,
any solution $f:U\to\C$ of the Beltrami equation 
\begin{equation}\label{eq:Beltrami}
	f_{\bar z} = \mu_b f_z
\end{equation}
is a holomorphic map from $(U,J_b)$ to $(\C,\Jst)$.
(If $U$ is a domain in $\C$ and $\|\mu_b\|_{U,\infty}<1$,
where $\|\cdotp\|_\infty$ is the essential supremum,   
a map $f$ satisfying \eqref{eq:Beltrami} is also called 
$\mu_b$-quasiconformal.) 
The theory of quasiconformal maps on domains in $\C$ was
developed by Ahlfors and Bers \cite{AhlforsBers1960},
also for Beltrami multipliers in $L^p$ spaces for $p>1$.
(See the survey and references in \cite[Sec.\ 2]{Forstneric2024Runge}.) 
When $2<p<\infty$, any local solution of \eqref{eq:Beltrami}
is of H\"older class $\Cscr^\alpha$ with $\alpha=1-2/p$,
while for Beltrami multipliers of class $\Cscr^{(k,\alpha)}$
with $k\in\Z_+$ and $0<\alpha<1$, 
the equation \eqref{eq:Beltrami} has local solutions 
of class $\Cscr^{(k+1,\alpha)}$. Furthermore, solutions in these 
function spaces can be chosen to depend locally analytically on $\mu$,
so they are as regular in the parameter $b$ as the 
family $\mu_b$ (or, equivalently, $J_b$).

Ahlfors and Bers also developed the global theory of quasiconformal
maps on planar domains. 
By \cite[Theorem 6, p.\ 396]{AhlforsBers1960}, for any $\mu$
on the plane or the disc with $\|\mu\|_{\infty}<1$ 
there exist unique $\mu$-conformal homeomorphisms of the whole plane 
and the unit disk onto themselves with fixed points at 0, 1, $\infty$ 
and 0,1 respectively. We state the following special case for later reference.

%
%
\begin{proposition}\label{prop:disc}
Given a continuous family $\{J_b\}_{b\in B}$ of complex structures 
of class $\Cscr^\alpha(\bar \D)$ on the closed disc $\bar\D$,  
there is a unique continuous family of diffeomorphisms 
$\phi_b:\bar \D \to \bar \D$ of class $\Cscr^{(1,\alpha)}(\bar \D)$ such 
that $\phi_b:\D\to\D$ is $J_b$-holomorphic and satisfies 
$\phi_b(0)=0,\ \phi_b(1)=1$ for every $b\in B$.
\end{proposition}

A similar statement holds for a continuous family 
$\{J_b\}_{b\in B}$ of complex structures 
of class $\Cscr^\alpha$ on $\R^2$ which are quasiconformally
equivalent to the standard structure $\Jst$ on $\C$.

The Ahlfors--Bers theory extends to smoothly bounded relatively
compact domains in smooth open surfaces.
If $\Jscr=\{J_b\}_{b\in B}$ is a continuous family
of complex structures of local class $\Cscr^{(k,\alpha)}$ on a smooth
open surface $X$ and $\Omega$ is a smoothly bounded relatively
compact domain in $X$, then every point $b_0\in B$ has a neighbourhood
$B_0\in B$ and a continuous family of $(J_b,J_{b_0})$-holomorphic
diffeomorphisms $\Phi_b:\Omega\to \Phi_b(\Omega)\subset X$,
$b\in B_0$, of class $\Cscr^{(k+1,\alpha)}(\Omega)$ (see  
\cite[Corollary 4.5]{Forstneric2024Runge}).
This result plays a major role in the proof of the Oka principle
for $\Jscr$-holomorphic maps $F:B\times X\to Y$ to any 
Oka manifold $Y$ (see \cite[Theorem 1.6]{Forstneric2024Runge}).
In turn, this Oka principle is the underlying tool to prove 
the results in this paper. It will be applied for maps to the 
punctured null quadric 
\begin{equation}\label{eq:nullq}
	\boldA =\bigl\{(z_1,\ldots,z_n)\in \C^n_* = \C^n\setminus \{0\}: 
	z_1^2+z_2^2+\cdots + z_n^2=0\bigr\},\quad n\ge 3,
\end{equation}
and also to some other Oka manifolds. 
(For the theory of Oka manifolds, see \cite[Chap.\ 5]{Forstneric2017E}.)

\subsection{Holomorphic null immersions and conformal minimal immersions}\label{ss:CMI}
We recall the basics on holomorphic
null immersions and conformal minimal immersions
of open Riemann surfaces to $\C^n$ and $\R^n$, respectively. 
For more information, see the monographs
\cite{Osserman1986,AlarconForstnericLopez2021}, among
many other sources. 

A holomorphic immersion $F=(F_1,F_2,\ldots,F_n):X\to \C^n$, $n\ge 3$, from 
an open Riemann surface $X$ is said to be a null immersion if its differential
$dF$ satisfies the nullity condition 
\[
	(dF_1)^2 + (dF_2)^2 + \cdots + (dF_n)^2 =0.
\]
Choosing a nowhere vanishing holomorphic $1$-form 
$\theta$ on $X$ (such exists by \cite{GunningNarasimhan1967}),
we have $dF=f\theta$ where $f=(f_1,\ldots,f_n):X\to \boldA$
is a holomorphic map to the punctured null quadric 
\eqref{eq:nullq}. Conversely, fixing a point $x_0\in X$,
a holomorphic map $f:X\to \boldA$ 
such that $f\theta$ is an exact $1$-form on $X$ determines a 
holomorphic null immersion $F:X\to\C^n$ by
\begin{equation}\label{eq:nullcurve}
	F(x)=F(x_0)+\int_{x_0}^x f\theta,  \qquad x\in X.
\end{equation}
An immersion $u=(u_1,u_2,\ldots,u_n):X\to\R^n$, $n\ge 3$, is conformal 
if and only if its $(1,0)$-differential $\di u=(\di u_1,\di u_2,\ldots,\di u_n)$ 
satisfies
\[
	(\di u_1)^2 + (\di u_2)^2 + \cdots + (\di u_n)^2=0.
\]
A conformal immersion $u$ is minimal if and only
if it is harmonic, 
\[
	dd^c u = -2\imath \dibar\di u=0.
\] 
In this case, $2\di u=f\theta$ where $f: X\to \boldA$ is 
a holomorphic map, and we recover $u$ from $f$ by 
\begin{equation}\label{eq:CMI}
	u(x)= u(x_0) + \Re \int_{x_0}^x f\theta,\qquad x\in X.
\end{equation}
The formulas \eqref{eq:nullcurve} and \eqref{eq:CMI} are 
known as the {\em Weierstrass representation} of 
holomorphic null curves and conformal minimal surfaces,
respectively, and $f\theta$ is called the Weierstrass data 
of $F$ or $u$.

In the special case $n=3$ considered in this paper, a 
more explicit Weierstrass representation formula for 
holomorphic null immersions $F:X\to\C^3$ 
with $dF=f\theta=(f_1,f_2,f_3)\theta$ is given by
\begin{equation}\label{eq:null3}
	F(x) = F(x_0)+ \int_{x_0}^x 
	\Big(1,\frac12\Big(g-\frac{1}{g}\Big),
	\frac{\imath}{2} \Big(g+\frac{1}{g}\Big) \Big) f_1\theta,
\end{equation}
where $g:X\to \CP^1$ is, up to a change of coordinates,
the (holomorphic) Gauss map of $F$.
(See \cite[p.\ 101 and Sec.\ 2.5]{AlarconForstnericLopez2021}
for the details.) The real part of the above integral gives a formula
reproducing conformal minimal surfaces $u:X\to\R^3$
from its Weierstrass data. 

Given a continuous family $\Jscr=\{J_b\}_{b\in B}$ of complex structures
on $X$, there is a continuous family $\{\theta_b\}_{b\in B}$ of 
nowhere vanishing $J_b$-holomorphic 1-forms on $X$
(see \cite[Theorem 7.1]{Forstneric2024Runge}). 
Using such a family, the analogous representation formulas 
hold for families of $\Jscr$-holomorphic null curves and conformal
minimal surfaces with continuous dependence on $b\in B$. 

Let us record the following consequence of Proposition \ref{prop:disc}.

%
%
\begin{proposition}\label{prop:CMIdisc}
Assume that $u:\D\to \R^n$ is a conformal minimal immersion.
Given a continuous family of complex structures $\Jscr=\{J_b\}_{b\in B}$
of class $\Cscr^\alpha(\bar \D)$ on $\bar\D$, there is a 
continuous family of $J_b$-conformal harmonic immersions 
$u_b:\D\to\R^n$ satisfying $u_b(\D)=u(\D)$ for all $b\in B$. 
\end{proposition}

\begin{proof}
Letting $\phi_b:\D\to\D$ be a family of $J_b$-holomorphic
diffeomorphisms given by Proposition \ref{prop:disc}, the 
family $u_b=u\circ \phi_b:\D\to\R^n$, $b\in B$, clearly satisfies 
Proposition \ref{prop:CMIdisc}. The analogous conclusion 
holds for null curves $\D\to\C^n$.
\end{proof}

%
%
\subsection{Convex integration lemmas}\label{ss:CI}
We shall use the parametric version of 
Gromov's {\em convex integration lemma}; 
see Gromov \cite[2.1.7.\ One-Dimensional Lemma]{Gromov1973}
for the basic case and Spring \cite[Theorem 3.4, p.\ 39]{Spring1998}
for the parametric case. 
By $\Co(A)$ we denote the convex hull of a subset $A\subset\R^m$.
The first lemma concerns the punctured null quadric 
$\boldA\subset\C^{n}$, $n\ge 3$ \eqref{eq:nullq}.

\begin{lemma}\label{lem:CI1}
Let $B$ be a compact Hausdorff space, 
$f_0:B\times [0,1]\to \boldA$ a continuous map, 
and $F_0: B\times [0,1]\to \C^n$ a continuous map such that 
$F_0(b,\cdotp):[0,1]\to \C^n$ is of class $\Cscr^1([0,1])$, 
its $t$-derivative $\dot F_0(b,t)=\frac{\di}{\di t}F_0(b,t)$ 
depends continuously on $(b,t)\in B\times [0,1]$, and 
\[
	\dot F_0(b,0)=f_0(b,0), \quad 
	\dot F_0(b,1)=f_0(b,1)\quad \text{for all $b\in B$}.
\]
Given $\epsilon >0$ there is a continuous map
$F:B \times  [0,1] \to \C^n$ satisfying the following for all $b\in B$:
\begin{enumerate}[\rm (a)]
\item In the endpoints $t=0,1$ of $[0,1]$, the values of 
$F(b,\cdotp)$ and its first $t$-derivatives coincide 
with those of $F_0(b,\cdotp)$.
\item $\|F-F_0\|_\infty<\epsilon$. 
\item $f(b,t):=\dot F(b,t) \in \boldA$ for all $t\in [0,1]$.
\end{enumerate}
\end{lemma}

\begin{proof}
When $B=\{b_0\}$ is a singleton and $\boldA$ is replaced by an
open subset $\Omega\subset \C^n$ with $\Co(\Omega)=\C^n$,
the lemma coincides with \cite[2.1.7]{Gromov1973}. 
The parametric case follows from \cite[Theorem 3.4, p.\ 39]{Spring1998}.
See also \cite[proof of Lemma 3.1]{ForstnericLarusson2019CAG},
where it is shown how to replace an open subset $\Omega\subset \C^n$ 
in the stated result with the submanifold $\bold A\subset \C^n$.
The idea of proof is to let $\Omega$ be a small open tubular neighbourhood
of $\bold A$, with a smooth retraction $\rho:\Omega\to \boldA$. 
Hence, $\Co(\Omega)=\C^n$. 
Applying \cite[Theorem 3.4, p.\ 39]{Spring1998} to maps 
$B\times [0,1]\to \Omega$ and postcomposing them by the retraction 
$\rho:\Omega\to \bold A$ gives a map $f:B\times [0,1]\to \boldA$,
obtained by suitably deforming the map $f_0$ with fixed ends at $t=0,1$,
such that the map $F:B\times [0,1]\to \C^n$ defined by 
\[
	F(b,t)= F_0(b,0) + \int_{0}^t f(b,s)\, ds,\quad t\in [0,1],\ b\in B
\]
satisfies the lemma, except that $F(b,1)$ is only close to $F_0(b,1)$ 
but not necessarily equal to it. 
As shown in \cite[Proof of Theorem 1.1]{ForstnericLarusson2019CAG}, 
the identity $F(b,1)=F_0(b,1)$ for all $b\in B$ 
can be obtained by a suitable correction, using period dominating 
sprays on a family of nondegenerate curves $t\mapsto f(b,t)\in \boldA$.
It is also explained in \cite[Remark 3.2]{ForstnericLarusson2019CAG}
why we can use any compact Hausdorff space $B$ as the parameter space,
as opposed to merely manifolds of class $\Cscr^1$
in the hypothesis of \cite[Theorem 3.4, p.\ 39]{Spring1998}.
\end{proof}

Our second lemma pertains to maps $F(b,t)$ 
whose $t$-derivatives belong to a variable family of closed 
algebraic submanifolds of $\C^n$.

\begin{lemma}\label{lem:CI2}
Let $B$ be a compact Hausdorff space, $\wt B=B\times [0,1]$,
and $\{A_{b,t}: b\in B,\ t\in [0,1]\}$ a continuous family of
connected, closed, algebraic submanifolds of $\C^n$, $n\ge 2$,  
none of them contained in any affine hyperplane of $\C^n$. 
Denote by $\pi:Z=\wt B\times \C^n \to \wt B$ the natural projection,
and let $\Acal\subset Z$ be the subset whose fibre over $(b,t)\in \wt B$
equals $A_{b,t}$. Let $f_0:\wt B \to \Acal\subset Z$ 
and $F_0:\wt B\to Z$ be continuous sections of $\pi:Z\to\wt B$ 
such that $F_0(b,\cdotp)$ is of class $\Cscr^1([0,1])$, 
its $t$-derivative $\dot F_0(b,t)=\frac{\di}{\di t}F_0(b,t)\in\C^n$ 
depends continuously on $(b,t)\in \wt B$, and 
\[
	\dot F_0(b,0)=f_0(b,0), \quad 
	\dot F_0(b,1)=f_0(b,1)  \quad \text{for all $b\in B$}.
\]
Given $\epsilon >0$ there is a section 
$F:B \times  [0,1] \to Z$ of the same class as $F_0$
and satisfying the following conditions for all $b\in B$.
\begin{enumerate}[\rm (a)]
\item $F(b,0)=F_0(b,0)$, $\dot F(b,0)=f_0(b,0)$, and 
$\dot F(b,1)=f_0(b,1)$.
\item $\|F-F_0\|_\infty<\epsilon$. 
\item $f(b,t):=\dot F(b,t) \in A_{b,t}$ for all $t\in [0,1]$.
\end{enumerate}
\end{lemma}

\begin{proof}
By \cite[Lemma 3.5.1]{AlarconForstnericLopez2021}, 
the conditions on the submanifolds $A_{b,t}\subset\C^n$ imply 
$\Co(A_{b,t})=\C^n$ for every $(b,t)\in \wt B$. 
Hence, the lemma is a special case of \cite[Theorem 3.4, p.\ 39]{Spring1998},
except that we must replace the subset $\Acal\subset Z$ by 
a small open neighbourhood $\Omega\subset Z$ with a 
fibre preserving retraction $\Omega\to A$.
The idea of proof is the same as for Lemma \ref{lem:CI1}.
(In this case, we do not need to achieve the exact
conditions $F(b,1)=F_0(b,1)$, $b\in B$.) 
\end{proof}

%
%
\begin{remark}\label{rem:LemmaCI}
In the present paper, Lemma \ref{lem:CI2} will be used for families of 
nonsingular quadric curves 
\[
	A_{b,t}=\{(x,y)\in\C^2: x^2+y^2=c(b,t)\} 
\]
with a continuous function $c:B \times  [0,1] \to\C^*$; 
see Claim \ref{cl:integral}.
\end{remark}

%
%
%
%
\section{Proof of Theorem \ref{th:CMI}}\label{sec:outline}

%
%
\subsection{Preparations}\label{ss:preparations}
For $l\in\N=\{1,2,\ldots\}$, $l\ge 2$, we denote $\Z_l=\Z/l\Z=\{0,1,\ldots,l-1\}$. Let $C\subset X$ be a (closed) Jordan curve. By a {\em division} of $C$ we mean a family of compact connected subarcs $\Dscr=\{\alpha_a: a\in\Z_l\}$ $(l\ge 3)$ of $C$ such that $\bigcup_{a\in\Z_l}\alpha_a=C$, the subarcs $\alpha_a$ and $\alpha_{a+1}$ have a common endpoint and are otherwise disjoint for every $a\in\Z_l$, and $\alpha_a\cap\alpha_e=\varnothing$ for every $a,e\in\Z_l$ with $e\notin\{a-1,a,a+1\}$. Assume now that $\Ccal=\bigcup_{i=1}^kC_i\subset X$ is a union of finitely many pairwise disjoint Jordan curves. By a division of $\Ccal$, we mean a family $\Dscr=\bigcup_{i=1}^k\Dscr_i$ where $\Dscr_i$ is a division of $C_i$ for $i=1,\ldots,k$; here the number of subarcs in $\Dscr_i$ might depend on $i$. Finally, assume that $K\subset X$ is a smoothly bounded compact domain with the boundary components $C_1,\ldots,C_k$, so $bK=\Ccal$. If $\delta>0$ is a number and $v:B\times K\to\R^3$ is a continuous map such that $v_b=v(b,\cdot)=(v_{b,1},v_{b,2},v_{b,3}):K\to\R^3$ is a $J_b$-conformal minimal immersion (on a neighbourhood of $K$) such that 
\[
	\max\{v_{b,1}(p),v_{b,2}(p)\}>\delta\quad 
	\text{for every $p\in bK$ and $b\in B$},
\] 
then we say that a division $\Dscr=\bigcup_{i=1}^k\Dscr_i$ of 
$bK=\bigcup_{i=1}^kC_i$ is {\em compatible with $(v,\delta)$} if the following conditions are satisfied for every $i\in\{1,\ldots,k\}$.
\begin{enumerate}[({$\mathfrak D$}1)]
\item The division $\Dscr_i$ of $C_i$ is of the form $\Dscr_i=\{\alpha_{i,a}: a\in\Z_{l_i}\}$ for an even integer $l_i\ge 4$.
\item $v_{b,1}(p)>\delta$ for every $p\in\alpha_{i,a}$, $a\in\Z_{l_i}$ odd, and $b\in B$.
\item $v_{b,2}(p)>\delta$ for every $p\in\alpha_{i,a}$, $a\in\Z_{l_i}$ even, and $b\in B$.
\item If $a\in \Z_{l_i}$ and $p_{i,a}$ denotes the only point in $\alpha_{i,a}\cap\alpha_{i,a+1}$, then 
$\partial_{J_b} v_{b,1}(p_{i,a})\neq 0$ 
and $\partial_{J_b} v_{b,2}(p_{i,a})\neq 0$ for every $b\in B$.
\end{enumerate}
These conditions imply $\min\{v_{b,1}(p_{i,a}),v_{b,2}(p_{i,a})\}>\delta$ for every $i\in\{1,\ldots,k\}$, $a\in\Z_{l_i}$, and $b\in B$. In general, given $v$ and $\delta$ as above, a division of $bK$ compatible with $(v,\delta)$ 
need not exist, but it always exists when $B$ is a singleton. This is an important extra difficulty with respect to the construction of isolated proper minimal surfaces.

We shall need the following notions.

\begin{definition}[\text{\cite[Definitions 1.12.9 and 3.1.2]{AlarconForstnericLopez2021}}]
\label{def:admissible}
An {\em admissible} set in $X$ is a compact set of the form $S=K\cup E$, where $K$ is a (possibly empty) finite union of pairwise disjoint compact domains with piecewise $\Cscr^1$ boundaries in $X$ and $E = \overline{S\setminus K}$ is a union of finitely many pairwise disjoint smooth Jordan arcs and closed Jordan curves meeting $K$ only at their endpoints (if at all) and such that their intersections with the boundary $bK$ of $K$ are transverse.

Let $J$ be a complex structure on $X$ and $\theta$ be a nowhere vanishing holomorphic $1$-form on $X$. A {\em generalized $J$-conformal minimal immersion} $S\to\R^3$ of class $\Cscr^r$ $(r\in\N)$ is a pair $(u,f\theta)$, where $u:S\to\R^3$ is a $\Cscr^r$ map whose restriction to $\mathring S=\mathring K$ is a $J$-conformal minimal immersion, and $f:S\to \boldA\subset\C^3$ \eqref{eq:nullq} is a map of class $\Cscr^{r-1}$, $J$-holomorphic on $\mathring S=\mathring K$, such that $f\theta=2\di_J u$ everywhere on $K$, and for any smooth path $\beta$ in $X$ parametrizing a component of $E$ we have that $\Re(\beta^*(f\theta))=\beta^*(du)=d(u\circ\beta)$.
Similarly, a {\em generalized $J$-holomorphic null immersion} 
$S\to\C^3$ of class $\Cscr^r$ is a pair $(F,f\theta)$,
where $f$ is as above and $F:S\to\C^n$ is a $\Cscr^r$ map
$(r\in\N)$ satisfying $dF=f\theta$ which is $J$-holomorphic 
on $\mathring S$.
\end{definition}

%
%
\subsection{Outline of proof of Theorem \ref{th:CMI}}
Let $\Fcal_b:H_1(X,\Z)\to\R^3$, $b\in B$, be a continuous family of homomorphisms. We shall construct a map $u=(u_1,u_2,u_3):B\times X\to \R^3$ satisfying the theorem, with $\Flux_{u_b}=\Fcal_b$
and $(u_{b,1},u_{b,2}):X\to\R^2$ being proper
for every $b\in B$, in an inductive process. 
To begin with, choose an exhaustion
\begin{equation}\label{eq:exhaustion}
	X_1\subset X_2\subset\cdots\subset \bigcup_{j\ge 0}X_j=X
\end{equation}
of $X$ by connected, smoothly bounded Runge compact domains such that $X_1$ is a disc, $X_j\subset\mathring X_{j+1}$ for every $j\in \N$, and the Euler characteristic $\chi(X_{j+1}\setminus\mathring X_j)$ of $X_{j+1}\setminus\mathring X_j$ is either $0$ or $-1$ for every $j\in\N$ (such exists by standard topological arguments; see e.g. \cite[Lemma 4.2]{AlarconLopez2013JGEA}). The latter condition simply means that the smoothly bounded compact domain $X_{j+1}\setminus\mathring X_j$ consists of finitely many connected components all of them being annuli, except perhaps one that might be a pair of pants, that is, a sphere with three discs removed. Hence, $X_j$ is a strong deformation retract of $X_{j+1}$ when $\chi(X_{j+1}\setminus\mathring X_j)=0$ (the noncritical case), while the number of boundary components of $X_j$ and $X_{j+1}$ differ by one when $\chi(X_{j+1}\setminus\mathring X_j)=-1$ (the critical case).

We claim that there is a sequence 
$\Xi_j=\{u^j,\Dscr^j,\epsilon_j\}$, $j\in\N$, where
\begin{itemize}
\item $u^j:B\times X_j\to\R^3$ is a continuous map such that $u^j_b=u^j(b,\cdot)=(u^j_{b,1},u^j_{b,2},u^j_{b,3}):X_j\to\R^3$ is a nonflat $J_b$-conformal minimal immersion (on a neighbourhood of $X_j$) with 
\begin{equation}\label{eq:max>j}
	\max\{u^j_{b,1}(p),u^j_{b,2}(p)\}>j\quad 
	\text{for every $p\in bX_j$ and $b\in B$},
\end{equation}
\item $\Dscr^j$ is a division of $bX_j$ compatible with $(u^j,j)$, and
\item $0<\epsilon_j<1$ is a number,
\end{itemize}
such that the following conditions hold for every $j\ge 2$:
\begin{enumerate}[(a$_j$)]
\item $\max\{u^j_{b,1}(p),u^j_{b,2}(p)\}>j-1$ for every 
$p\in X_j\setminus\mathring X_{j-1}$ and $b\in B$.
\item $|u^j_b(p)-u^{j-1}_b(p)|<\epsilon_{j-1}$ for every $p\in X_{j-1}$ and $b\in B$.
\item $\epsilon_j<\epsilon_{j-1}/2$, and if $v:X\to\R^3$ is a $J_b$-conformal harmonic map such that $|v(p)-u^j_b(p)|<2\epsilon_j$ for every $p\in X_j$ and some $b\in B$, then $v$ is a nonflat immersion on $X_j$. 
\item $\Flux_{u^j_b}(C)=\Fcal_b(C)$ for every closed curve $C\subset X_j$ and $b\in B$.
\end{enumerate}
We emphasize that the existence of each division $\Dscr^j$ is established as a part of the induction.

Assume that such a sequence exists. By \eqref{eq:exhaustion}, (b$_j$), and the first part of (c$_j$), there is a limit continuous map
\[
	u=\lim_{j\to\infty} u^j:B\times X\to\R^3
\]
such that $u_b=u(b,\cdot)=(u_{b,1},u_{b,2},u_{b,3}):X\to\R^3$ is a $J_b$-conformal harmonic map with
\begin{equation}\label{eq:<2epsilonj}
	|u_b(p)-u^j_b(p)|<2\epsilon_j\quad \text{for every $p\in X_j$, $j\in\N$, and $b\in B$}.
\end{equation}
The second part of (c$_j$) then guarantees  that $u_b$ is a nonflat $J_b$-conformal minimal immersion, while conditions (d$_j$) ensure that $\Flux_{u_b}=\Fcal_b$ holds for every $b\in B$; take \eqref{eq:exhaustion} into account. Furthermore, given $b\in B$ and $j\ge 2$, we have by (a$_j$), (b$_j$), and \eqref{eq:<2epsilonj} that
\[
	\max\{u_{b,1}(p),u_{b,2}(p)\}>j-1-2\epsilon_j>j-3
	\quad \text{for every $p\in X_{j+1}\setminus\mathring X_j$}.
\]
This and \eqref{eq:exhaustion} imply that $\max\{u_{b,1},u_{b,2}\}:X\to\R$ is a proper map, hence so is $(u_{b,1},u_{b,2}):X\to\R^2$. Therefore, the map $u$ satisfies the conclusion of the theorem. 

To complete the proof, it remains to 
construct a sequence $\Xi_j$ with the desired properties.

%
%
\subsection{The induction}\label{ss:induction}
Set $\epsilon_0=1$ and $X_0=\varnothing$. For the base case when $j=1$, choose any nonflat $\Jst$-conformal minimal immersion $v=(v_1,v_2,v_3):\bar\D\to\R^3$ such that 
\begin{equation}\label{eq:basecase}
	\text{$v_i(p)>1$ and $\partial_{\Jst} v_i(p)\neq 0$  for every  $p\in\bar\D$ and $i\in\{1,2\}$}. 
\end{equation}
Since $X_1$ is a disc, Proposition \ref{prop:CMIdisc} furnishes a continuous map $u^1:B\times X_1\to\R^3$ such that 
$u^1_b=u^1(b,\cdot)=(u^1_{b,1},u^1_{b,2},u^1_{b,3}):X_1\to\R^3$ is a nonflat $J_b$-conformal minimal immersion with $u^1_b(X_1)=v(\bar\D)$ for every $b\in B$. It follows from \eqref{eq:basecase} that condition (a$_1$) and inequality \eqref{eq:max>j} for $j=1$ are satisfied, while any division $\Dscr^1$ of the Jordan curve $bX_1$ consisting of (for instance) $4$ subarcs is compatible with $(u^1,1)$; see conditions ({$\mathfrak D$}1) to ({$\mathfrak D$}4) in Subsection \ref{ss:preparations}. Furthermore, (d$_1$) holds true since $X_1$ is simply connected. Finally, by compactness of $B$, the Cauchy estimates provide a number $\epsilon_1>0$ satisfying condition (c$_1$), whereas condition (b$_1$) is void.

For the inductive step, fix an integer $j\ge 2$ and assume that we already have a triple $\Xi_{j-1}=\{u^{j-1},\Dscr^{j-1},\epsilon_{j-1}\}$ as above, satisfying condition (d$_{j-1}$) and inequality \eqref{eq:max>j} with $j-1$ in place of $j$. 
We shall construct a suitable triple $\Xi_j=\{u^j,\Dscr^j,\epsilon_j\}$ satisfying conditions (a$_j$) to (d$_j$) as well as inequality \eqref{eq:max>j}. For this, we distinguish cases depending on the Euler characteristic of $X_j\setminus \mathring X_{j-1}$.

%
%
\medskip
\noindent{\bf The noncritical case.} Assume that $\chi(X_j\setminus\mathring X_{j-1})=0$, so $X_{j-1}$ is a strong deformation retract of $X_j$. For simplicity of exposition, we assume that $bX_j$ is connected, hence $bX_{j-1}$ is connected as well and $X_j\setminus \mathring X_{j-1}$ is a smoothly bounded compact annulus. For the general case, we apply the same procedure in each component of $X_j\setminus \mathring X_{j-1}$.

Write $\Dscr^{j-1}=\{\alpha_a: a\in\Z_l\}$, $l\ge 4$ even, for the given division of the closed Jordan curve $bX_{j-1}$ compatible with $(u^{j-1},j-1)$, and denote by $p_a$ the only point in $\alpha_a\cap\alpha_{a+1}$, $a\in\Z_l$. Set
\begin{equation}\label{eq:Zlodd}
	\Z_l^{\rm odd}=\{a\in\Z_l: a\text{ odd}\}
	\quad\text{and}\quad
	\Z_l^{\rm even}=\{a\in\Z_l: a\text{ even}\}.
\end{equation}
It is obvious that $\Z_l=\Z_l^{\rm odd}\cup\Z_l^{\rm even}$ and $\Z_l^{\rm odd}\cap\Z_l^{\rm even}=\varnothing$. Conditions ({$\mathfrak D$}2) to ({$\mathfrak D$}4) then give the following properties for every $b\in B$.
\begin{enumerate}[({I}$_1$)]
\item $u^{j-1}_{b,1}(p)>j-1$ for every $p\in\alpha_a$, $a\in\Z_l^{\rm odd}$.
\item $u^{j-1}_{b,2}(p)>j-1$ for every $p\in\alpha_a$, $a\in\Z_l^{\rm even}$.
\item $\partial_{J_b} u^{j-1}_{b,1}(p_a)\neq 0$ and $\partial_{J_b} u^{j-1}_{b,2}(p_a)\neq 0$ for every $a\in\Z_l$.
\end{enumerate}
In particular, 
\begin{equation}\label{eq:min}
	\min\{u^{j-1}_{b,1}(p_a),u^{j-1}_{b,2}(p_a)\}>j-1\quad \text{for every $a\in\Z_l$, $b\in B$.} 
\end{equation}
\smallskip
\noindent{\em The first deformation.} In the first step of the construction,  we shall approximate $u^{j-1}$ uniformly on $B\times X_{j-1}$ by a continuous family $\hat u:B\times X_j\to\R^3$ of $J_b$-conformal minimal immersions with some control of their images. For each $a\in\Z_l$ we choose a smooth embedded arc $\gamma_a\subset X_j\setminus\mathring X_{j-1}$ with the initial point $p_a\in bX_{j-1}$, the final point $q_a\in bX_j$, and otherwise disjoint from $bX_j\cup bX_{j-1}$. We choose the family of arcs $\gamma_a$, $a\in\Z_l$, to be pairwise disjoint and such that the set 
\[
	S=X_{j-1}\cup\bigcup_{a\in\Z_l}\gamma_a
\]  
is admissible in the sense of Definition \ref{def:admissible}. Also, for each $a\in \Z_l$ we denote by $\beta_a$ the arc in $bX_j$ connecting $q_{a-1}$ and $q_a$ without meeting any of the other points $q_e$ for $e\in\Z_l\setminus\{a-1,a\}$. We denote by $D_a\subset X_j\setminus\mathring X_{j-1}$ the closed disc bounded by $\gamma_{a-1}$, $\alpha_a$, $\gamma_a$, and $\beta_a$, $a\in\Z_l$.
(See Figure \ref{fig:Da}.) Note that $D_a\cap D_{a+1}=\gamma_a$, $D_a\cap D_e=\varnothing$ for $e\in\Z_l\setminus\{a-1,a,a+1\}$,
\begin{equation}\label{eq:cupDa}
	X_j\setminus\mathring X_{j-1}=\bigcup_{a\in\Z_l}D_a =\bigcup_{a\in\Z_l^{\rm odd}}D_a\cup \bigcup_{a\in\Z_l^{\rm even}}D_a,
\end{equation}
and
\begin{equation}\label{eq:cupbetaa}
	bX_j=\bigcup_{a\in\Z_l}\beta_a =\bigcup_{a\in\Z_l^{\rm odd}}\beta_a\cup \bigcup_{a\in\Z_l^{\rm even}}\beta_a.
\end{equation}
%
%
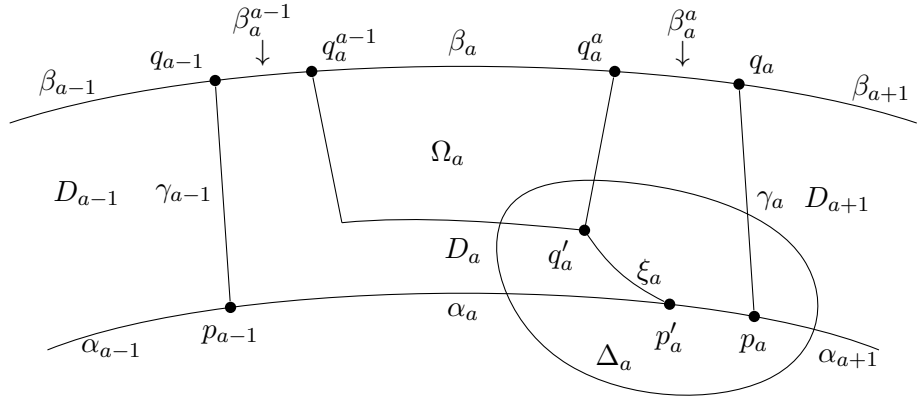
\begin{figure}[htpb]
    \centering
    \begin{tikzpicture}[scale=1, every node/.style={scale=1}]
        
        \draw (0,4) .. controls (3,5) and (9,5) .. (12,4)
          node[pos=0.08, above] {$\beta_{a-1}$}
          node[pos=0.25, circle, fill, inner sep=1.5pt] (qa1) {}
          node[pos=0.25, above left] {$q_{a-1}$}
          node[pos=0.3, above, align=center] {$\beta_a^{a-1}$ \\[-0.5ex] $\downarrow$}
          node[pos=0.35, circle, fill, inner sep=1.5pt] (qaa1) {}
          node[pos=0.35, above right] {$q_a^{a-1}$}
          node[pos=0.5, above] {$\beta_a$}
          node[pos=0.65, circle, fill, inner sep=1.5pt] (qaa) {}
          node[pos=0.65, above left] {$q_a^a$}
          node[pos=0.72, above, align=center] {$\beta_a^a$ \\[-0.5ex] $\downarrow$}
          node[pos=0.78, circle, fill, inner sep=1.5pt] (qa) {}
          node[pos=0.78, above right] {$q_a$}
          node[pos=0.95, above] {$\beta_{a+1}$};

        \draw (0.5,1) .. controls (3,2) and (9,2) .. (11.5,1)
          node[pos=0.1, below] {$\alpha_{a-1}$}
          node[pos=0.25, circle, fill, inner sep=1.5pt] (pa1) {}
          node[pos=0.25, below=0.1cm] {$p_{a-1}$}
          node[pos=0.5, below] {$\alpha_a$}
          node[pos=0.72, circle, fill, inner sep=1.5pt] (ppa) {}
          node[pos=0.72, below=0.1cm] {$p'_a$}
          node[pos=0.82, circle, fill, inner sep=1.5pt] (pa) {}
          node[pos=0.82, below=0.1cm] {$p_a$}
          node[pos=0.95, below] {$\alpha_{a+1}$};

        \draw (qa1) -- (pa1) node[midway, left] {$\gamma_{a-1}$};
        \node at ($(qa1)!0.5!(pa1) + (-1.8, 0)$) {$D_{a-1}$};

        \draw (qa) -- (pa) node[midway, right] {$\gamma_a$};
        \node at ($(qa)!0.5!(pa) + (1.2, 0)$) {$D_{a+1}$};

        \coordinate (M) at ($(qaa1) + (0.4, -2.0)$); 
        
        \coordinate (qpa) at ($(qaa) + (-0.4, -2.1)$); 

        \draw (qaa1) -- (M);
        \draw (qaa) -- (qpa);

        \draw (M) .. controls ($(M) + (1, 0.1)$) and ($(qpa) + (-1, 0.1)$) .. (qpa) 
          node[midway, below=0.1cm] {$D_a$};

        \node[circle, fill, inner sep=1.5pt] at (qpa) {};
        \node[below left] at (qpa) {$q'_a$};

        \draw (qpa) to[bend right=15] node[midway, right=1mm] {$\xi_a$} (ppa);

        \node at (5.8, 3.6) {$\Omega_a$};

        \draw plot [smooth cycle, tension=0.8] coordinates { 
            (6.5, 2.7) 
            (8.3, 3.2) 
            (10.5, 2.2) 
            (10.0, 0.6) 
            (7.4, 0.8) 
        };
        \node at (8.0, 0.9) {$\Delta_a$};

    \end{tikzpicture}
    
    \caption{Configuration of sets in $X_j\setminus\mathring X_{j-1}$.}
    \label{fig:Da}
    
\end{figure}
%
%

Consider the $\Jscr$-holomorphic map $f:B\times X_{j-1}\to\bold A$ to 
the punctured null quadric \eqref{eq:nullq} given by 
\[
	f(b,\cdot)=2\partial_{J_b} u_b^{j-1}/\theta_b\quad
	\text{for every $b\in B$}, 
\]
and set $f_b=f(b,\cdot)=(f_{b,1},f_{b,2},f_{b,3}):X_{j-1}\to\boldA$.

%
%

The following claim is an immediate consequence of Lemma \ref{lem:CI1};
for granting the second condition, take into account \eqref{eq:min}.

\begin{claim}\label{cl:arcs}
The pair of maps $u^{j-1}:B\times X_{j-1}\to\R^3$ and $f:B\times X_{j-1}\to\boldA$ can be extended to continuous maps $u^{j-1}:B\times S\to\R^3$ and $f:B\times S\to\boldA$, respectively, such that, setting $u^{j-1}_b=u^{j-1}(b,\cdot)=(u^{j-1}_{b,1},u^{j-1}_{b,2},u^{j-1}_{b,3})$ and $f_b=f(b,\cdot)=(f_{b,1},f_{b,2},f_{b,3})$ for every $b\in B$, the following conditions are satisfied.
\begin{itemize}
\item The pair $(u^{j-1}_b,f_b\theta_b)$ is a generalized $J_b$-conformal minimal immersion $S\to\R^3$ of class $\Cscr^1$ for every $b\in B$ (see Definition \ref{def:admissible}).
\item $\min\{u^{j-1}_{b,1}(p),u^{j-1}_{b,2}(p)\}>j-1$ for every $p\in \bigcup_{a\in\Z_l}\gamma_a$ and $b\in B$.
\item $\min\{u^{j-1}_{b,1}(q_a),u^{j-1}_{b,2}(q_a)\}>j$ for every $a\in\Z_l$ and $b\in B$.
\item $f_{b,1}(q_a)f_{b,2}(q_a)\neq0$ for every  $a\in\Z_l$ and $b\in B$.
\end{itemize}
\end{claim}

By the Mergelyan theorem for families of conformal minimal immersions  (see \cite[Theorem 8.2 and Corollary 8.6]{Forstneric2024Runge}), we can approximate $u^{j-1}$ uniformly on $B\times S$ by a continuous map $\hat u:B\times X_j\to\R^3$ satisfying the following conditions 
for every $b\in B$.
\begin{enumerate}[\rm ({II}$_1$)]
\item The map $\hat u_b=\hat u(b,\cdot)=(\hat u_{b,1},\hat u_{b,2},\hat u_{b,3}):X_j\to\R^3$ is a $J_b$-conformal minimal immersion.
\item $|\hat u_b(p)-u^{j-1}_b(p)|<\epsilon_{j-1}/3$ for every $p\in X_{j-1}$.
\item $\min\{\hat u_{b,1}(p),\hat u_{b,2}(p)\}>j-1$ for every $p\in \bigcup_{a\in\Z_l}\gamma_a$.
\item $\min\{\hat u_{b,1}(q_a),\hat u_{b,2}(q_a)\}>j$ for every $a\in\Z_l$.
\item $\hat f_{b,1}(q_a)\hat f_{b,2}(q_a)\neq 0$ for every  $a\in\Z_l$, where 
\[
	\hat f_b=(\hat f_{b,1},\hat f_{b,2},\hat f_{b,3})
	=2\partial_{J_b}\hat u_b/\theta_b:X_j\to\boldA.
\]
\item $\Flux_{\hat u_b}(C)=\Fcal_b(C)$ for every closed curve 
$C\subset X_j$; take into account (d$_{j-1}$).
\item $\hat f_{b,1}(p_a)\hat f_{b,2}(p_a)\neq 0$ for every $a\in\Z_l$; see condition (I$_3$).
\end{enumerate}

Taking a look into the conditions required in the induction, (II$_3$) and (II$_4$) show that the immersions $\hat u_b$ assume suitable values on $\bigcup_{a\in\Z_l}\gamma_a$ (see inequality \eqref{eq:max>j} and condition (a$_j$)), but they need not do so on the complement of $\bigcup_{a\in\Z_l}\gamma_a$ in $X_j\setminus X_{j-1}$, so we have to keep working. 

\smallskip
\noindent{\em The second deformation I: a spray of L\'opez--Ros transformations.} We shall now deform the family $\hat u_b$ on $X_j\setminus X_{j-1}$ in order to obtain more control on its image over the sets $D_a$ for $a\in\Z_l^{\rm odd}$; see \eqref{eq:cupDa}.

Consider the $\Jscr$-holomorphic map $\hat f:B\times X_j\to\boldA$ given by $\hat f(b,\cdot)=\hat f_b:X_j\to\boldA$ for every $b\in B$. Our next task is to embed $\hat f$ as the core of a period dominating spray of maps $B\times X_j\to\boldA$, keeping fixed the first component function; this will enable us to kill the periods in a subsequent step. In the nonparametric case, this has typically been done using the Oka principle for sections of ramified holomorphic maps with Oka fibres (see \cite[Section 6.14]{Forstneric2017E}); a tool that is not available in the parametric framework. Instead, we shall use a parametric version of the L\'opez-Ros deformation for minimal surfaces (see \cite{LopezRos1991JDG}) as in \cite{AlarconLarusson2025JGA}. Set
\begin{equation}\label{eq:psi_b}
	\psi_b=\frac{\hat f_{b,1}}{\hat f_{b,2}-\imath\hat f_{b,3}}:
	X_j\to\CP^1,\quad b\in B.
\end{equation}
A calculation shows that 
\begin{equation}\label{eq:WFf1}
	\hat f_{b,2}=
	\frac12\Big(\frac1{\psi_b}-\psi_b\Big)\hat f_{b,1}
	\quad\text{and}\quad
	\hat f_{b,3}=
	\frac{\imath}2\Big(\frac1{\psi_b}+\psi_b\Big)\hat f_{b,1},\quad b\in B.
\end{equation}
(See \eqref{eq:null3}.) Moreover,
\begin{equation}\label{eq:dp}
	\frac{\hat f_{b,1}}{\psi_b}=\hat f_{b,2}-\imath\hat f_{b,3}
	\quad\text{and}\quad
	\psi_b\hat f_{b,1}=-\hat f_{b,2}-\imath\hat f_{b,3}
\end{equation}
are $J_b$-holomorphic functions on $X_j$ for every $b\in B$.
Note that for any $\Jscr$-holomorphic map 
$(\hat f_1,\tilde f_2,\tilde f_3):B\times X_j\to \boldA$
sharing the first component $\hat f_1$ with the map $\hat f$, 
there is a $\Jscr$-holomorphic function $\mu:B\times X_j\to\C^*$
such that 
\begin{equation}\label{eq:WFtildef1}
	\tilde f_{b,2}=
	\frac12\Big(\frac1{\mu_b\psi_b}-\mu_b\psi_b\Big)\hat f_{b,1},
	\quad 
	\tilde f_{b,3}=
	\frac{\imath}2\Big(\frac1{\mu_b\psi_b}+\mu_b\psi_b\Big)\hat f_{b,1},
	\quad b\in B.
\end{equation}

Fix a point $x_0\in\mathring X_{j-1}$ and let $\Gamma_1,\ldots,\Gamma_m$ be a family of smooth oriented Jordan curves in $\mathring X_{j-1}$ such that any two of them only intersect at $x_0$, they form a basis of the homology group $H_1(X_{j-1},\Z)$ (and hence also of $H_1(X_j,\Z)$), and the set $\Gamma=\bigcup_{i=1}^m \Gamma_i$ is Runge in $X$ \cite[Lemma 1.12.10]{AlarconForstnericLopez2021}. Denote by $\Pcal:B\times\Cscr(\Gamma,\C^*)\to(\C^2)^m$ the period map that sends for each $b\in B$ a map $h\in\Cscr(\Gamma,\C^*)$ to
\begin{equation}\label{eq:period-map}
	\Pcal(b,h)=\left( \int_{\Gamma_i} \Big(\frac{\hat f_{b,1}}{h\psi_b}
	\,,\,
	h\psi_b\hat f_{b,1}\Big)\theta_b\right)_{i=1,\ldots,m}\in (\C^2)^m.
\end{equation}
Since $\hat f_b$ is nonflat for each $b\in B$, the functions $\hat f_{b,1}/\psi_b$ and $\psi_b\hat f_{b,1}$ in \eqref{eq:dp} are complex linearly independent. We can therefore use \cite[Claim 2.3]{AlarconLarusson2025JGA} to obtain a continuous function
\[
	\sigma:\B\times B\times X_j\to\C^*,
\]
depending holomorphically on a parameter $\zeta=(\zeta_1,\ldots,\zeta_N)$ in a ball $\B\subset \C^N$ centred at the origin for some large $N\in\N$ and satisfying the following conditions
for every $b\in B$.
\begin{enumerate}[\rm ({III}$_1$)]
\item The function $\sigma(\zeta,b,\cdot):X_j\to\C^*$ is $J_b$-holomorphic for every $\zeta\in\B$.
\item $\sigma(0,b,p)=1$ for every $p\in X_j$.
\item $\sigma$ is period dominating, in the sense that the map
\[
	\B\ni\zeta\longmapsto 
	\Pcal\big(b,\sigma(\zeta,b,\cdot)\big)\in (\C^2)^m
\]
is submersive at $\zeta=0$; see \eqref{eq:period-map}.
\end{enumerate}

\begin{remark}
Note that \cite[Claim 2.3]{AlarconLarusson2025JGA} provides such a function in the special case when $J_b$ does not depend on $b\in B$, so it deals with a single open Riemann surface instead of with a family. Nevertheless, using also the argument in 
\cite[proof of Theorem 8.2, the noncritical case]{Forstneric2024Runge},
the proof in \cite{AlarconLarusson2025JGA} extends in a straightforward 
way to a variable family of complex structures on $X$. 
\end{remark}

Set $\sigma_{\zeta,b}=\sigma(\zeta,b,\cdot)$ for $\zeta\in\B$ and $b\in B$, and let $\varphi:\B\times B\times X_j\to\boldA$ be the continuous map 
\begin{equation}\label{eq:varphi}
	\varphi(\zeta,b,\cdot)=\left( \hat f_{b,1}\,,\, 
	\frac12\Big(\frac1{\sigma_{\zeta,b}\psi_b}-\sigma_{\zeta,b}\psi_b\Big)\hat f_{b,1}
	\,,\,
	\frac{\imath}2\Big(\frac1{\sigma_{\zeta,b}\psi_b}+\sigma_{\zeta,b}\psi_b\Big)\hat f_{b,1}\right),\quad (\zeta,b)\in\B\times B.
\end{equation}
It follows that $\varphi$ is well defined (that is, it takes its values in $\boldA$), and for every $\zeta\in\B$ and $b\in B$ the map $\varphi(\zeta,b,\cdot):X_j\to\boldA$ is $J_b$-holomorphic and satisfies $\varphi(0,b,\cdot)=\hat f_b$ .

\smallskip
\noindent{\em The second deformation II: replacement of the core.}  We shall now replace the functions $\psi_b$ in \eqref{eq:varphi} by suitable functions that are obtained by another application of the L\'opez--Ros transformation with parameters. The aim is to create a strong but controlled deformation of $\hat u_b$ on the sets $D_a$ for $a\in\Z_l^{\rm odd}$.

By (II$_7$) and compactness of $B$, for each $a\in\Z_l^{\rm odd}$ there is  an open, smoothly bounded disc neighborhood $\Delta_a$ of $p_a$ in $X_j$, with $\overline\Delta_a\cap (bX_j\cup\bigcup_{e\in\Z_l\setminus\{a\}}\gamma_e)=\varnothing$, such that
\begin{equation}\label{eq:Deltaa}
	\hat f_{b,1}(p)\hat f_{b,2}(p)\neq 0\quad
	\text{for every $p\in \overline\Delta_a$ and $b\in B$}.
\end{equation}
Further, by (II$_3$), (II$_4$), (II$_5$), and compactness of $B$, for each $a\in \Z_l^{\rm odd}$ there is a closed disc $\Omega_a\subset D_a\setminus(\gamma_{a-1}\cup\alpha_a\cup\gamma_a)$ such that the following conditions hold.
\begin{enumerate}[\rm ({IV}$_1$)]
\item $\Omega_a\cap\beta_a$ is a compact connected Jordan arc.
\item $\hat u_{b,1}(p)>j-1$ for every $p\in \overline{D_a\setminus\Omega_a}$ and $b\in B$.
\item $\hat u_{b,1}(p)>j$ for every $p\in \overline{\beta_a\setminus\Omega_a}$ and $b\in B$.
\item $\Delta_a\cap b\Omega_a\neq\varnothing$
and $D_a\cap (\Delta_a\setminus \mathring \Omega_a)$ is path connected. (The latter condition might require to replace $\Delta_a$ by a smaller neighbourhood of $p_a$.)
\item $\hat f_{b,1}(p)\hat f_{b,2}(p)\neq 0$ for every $p\in \overline{\beta_a\setminus \Omega_a}$ and $b\in B$.
\end{enumerate}
Denote by $\beta_a^{a-1}$ the connected component of $\overline{\beta_a\setminus\Omega_a}$ with $q_{a-1}$ as an endpoint, and $\beta_a^a$ the one containing $q_a$. Denote by $q_a^{a-1}$ and $q_a^a$ the other endpoint of $\beta_a^{a-1}$ and $\beta_a^a$, respectively. Clearly,
\begin{equation}\label{eq:betaa}
	\overline{\beta_a\setminus\Omega_a}=\beta_a^{a-1}\cup\beta_a^a.
\end{equation}
In view of (IV$_4$), for each $a\in \Z_l^{\rm odd}$ there is a smooth embedded arc $\xi_a\subset (D_a\setminus\mathring\Omega_a)\cap\Delta_a$ with initial point $p_a'\in \alpha_a\setminus\{p_a\}$ and endpoint $q_a'\in b\Omega_a$, and otherwise disjoint from 
$bD_a\cup b\Omega_a$. Choose these arcs so that the Runge compact set
\[
	S^{\rm odd}=X_{j-1}\cup \Big(\bigcup_{a\in\Z_l^{\rm even}}D_a \Big)\cup
	\Big(\bigcup_{a\in\Z_l^{\rm odd}}\Omega_a\cup\xi_a\Big)\subset X_j
\]
is admissible (Definition \ref{def:admissible}); see Figure \ref{fig:Da}.
Condition \eqref{eq:Deltaa} ensures that 
\begin{equation}\label{eq:Deltaa2}
	\hat f_{b,1}(p) \neq 0\quad
	\text{for every $b\in B$, $p\in \xi_a$, and $a\in\Z_l^{\rm odd}$}.
\end{equation}

%
%

\begin{claim}\label{cl:integral}
For any number $\tau>0$ there is a continuous function $\mu:B\times S^{\rm odd}\to\C^*$ such that, setting $\mu_b=\mu(b,\cdot)$, we have that $\mu_b(p)=1$ for every 
$p\in S^{\rm odd}\setminus\bigcup_{a\in\Z_l^{\rm odd}}\xi_a$ and
\[
	\Re\int_{\xi_a}\Big(\frac1{\mu_b\psi_b}-\mu_b\psi_b\Big)\hat f_{b,1}\theta_b>\tau
	\quad\text{for every $a\in\Z_l^{\rm odd}$ and $b\in B$}.
\]
\end{claim}

This follows from Lemma \ref{lem:CI2},
applied to the family of maps $(\hat f_{b,2},\hat f_{b,3}):\xi_a \to \C^2$
$(a\in\Z_l^{\rm odd},\ b\in B)$, together with the observation in \eqref{eq:WFtildef1}. Indeed, let $t\in [0,1]$ be a regular parameter on 
the arc $\xi_a$ for some $a\in\Z_l^{\rm odd}$. Since $\hat f_{b,1}(t)\ne 0$ 
for all $b\in B$ and $t\in [0,1]$ \eqref{eq:Deltaa2}, the equation 
$x^2+y^2=-\hat f_{b,1}(t)^2$
defines a noningular quadric curve $A_{b,t}\subset \C^2$,
and it remains to apply Lemma \ref{lem:CI2} to families of maps 
$[0,1]\ni t\mapsto (\tilde f_{b,2}(t),\tilde f_{b,3}(t)) \in A_{b,t}$ 
with $\tilde f_{b,i}(0)=\hat f_{b,i}(0)$ and $\tilde f_{b,i}(1)=\hat f_{b,i}(1)$ 
for $b\in B$ and $i=2,3$. The observation \eqref{eq:WFtildef1}
gives for each arc $\xi_a$, with $a\in\Z_l^{\rm odd}$, 
a continuous family of multipliers $\mu_{b,a}:\xi_a \to\C^*$
$(b\in B)$ assuming the value $1$ at the endpoints of 
$\xi_a$. Hence, we can extend these functions to 
$\mu_b: S^{\rm odd}\to\C^*$ taking the value $1$ 
on the complement of $\bigcup_{a\in\Z_l^{\rm odd}}\xi_a$.

Consider a function $\mu$ given by the claim for a fixed large number $\tau>0$ to be specified later. Since $\C^*$ is Oka, the Mergelyan theorem for families of holomorphic functions into Oka manifolds in \cite[Theorem 1.6]{Forstneric2024Runge} furnishes a continuous function $\hat \mu:B\times X_j\to\C^*$ satisfying the following.
\begin{enumerate}[\rm ({V}$_1$)]
\item The function $\hat \mu_b=\hat \mu(b,\cdot):X_{j-1}\to\C^*$ is $J_b$-holomorphic for every $b\in B$.
\item $|\hat \mu_b(p)-1|<1/\tau$ for every $p\in S^{\rm odd}\setminus\bigcup_{a\in\Z_l^{\rm odd}}\xi_a$ and $b\in B$.
\item $\displaystyle \Re\int_{\xi_a}\Big(\frac1{\hat\mu_b\psi_b}-\hat\mu_b\psi_b\Big)\hat f_{b,1}\theta_b>\tau$ for every $a\in\Z_l^{\rm odd}$ and $b\in B$.
\end{enumerate}

Assuming that $\tau>0$ is chosen sufficiently large, the implicit function theorem provides, in view of (V$_2$) and the period domination property of the spray $\sigma$ in (III$_3$), a continuous map $B\to\B$ sending $b\in B$ to $\zeta_b\in \B$ and satisfying the following conditions.
\begin{enumerate}[\rm ({VI}$_1$)]
\item $\zeta_b\in\B'$ for every $b\in B$, where $\B'\subset \B\subset\C^N$ is any given ball centered at the origin.
\item Setting $h_b=\sigma(\zeta_b,b,\cdot)\hat\mu_b$,
we have that $\Pcal\big(b,h_b\big)=\Pcal(b,1)$ for every $b\in B$; see \eqref{eq:period-map}.
\end{enumerate}
Consider the continuous map $\hat\psi:B\times X_j\to\CP^1$ 
given by 
\[
	\hat\psi_b=\hat\psi(b,\cdot)=h_b\psi_b
	=\sigma(\zeta_b,b,\cdot)\hat\mu_b\psi_b,\quad b\in B,
\]
see \eqref{eq:psi_b}, and the continuous map $\hat\varphi:B\times X_j\to\boldA$ given by
\begin{equation}\label{eq:hatvarphi}
	\hat\varphi(b,\cdot)=\left( \hat f_{b,1}\,,\, 
	\frac12\Big(\frac1{\hat\psi_b}-\hat\psi_b\Big)\hat f_{b,1}
	\,,\,
	\frac{\imath}2\Big(\frac1{\hat\psi_b}+\hat\psi_b\Big)\hat f_{b,1}\right),\quad b\in  B;
\end{equation}
cf.\ \eqref{eq:varphi} and \eqref{eq:null3}. It turns out that $\hat\varphi$ is well defined (that is, it takes its values in $\boldA$), and the map 
$\hat\varphi_b=(\hat\varphi_{b,1},\hat\varphi_{b,1},\hat\varphi_{b,1})=\hat\varphi(b,\cdot):X_j\to\boldA$ is $J_b$-holomorphic for every $b\in B$; that is, the map $\hat\varphi$ is $\Jscr$-holomorphic. Moreover, condition (VI$_2$) ensures that the $J_b$-holomorphic $1$-form $(\hat\varphi_b-\hat f_b)\theta_b$ is exact on $X_j$ for every $b\in B$ (see condition (II$_5$)). Consider the continuous map $u':B\times X_j\to\R^3$ given by
\[
	u'_b(p)=u'(b,p)=\hat u_b(x_0)
	+\Re\int_{x_0}^p\hat\varphi_b\theta_b
	\quad \text{ for $p\in X_j$ and $b\in B$},
\]
where $x_0\in \mathring X_{j-1}$ was fixed above. The following conditions hold for every $b\in B$.
\begin{enumerate}[\rm ({i}$_1$)]  
\item The map $u'_b=(u'_{b,1},u'_{b,2},u'_{b,3}):X_j\to\R^3$ is a $J_b$-conformal minimal immersion.
\item $u'_{b,1}=\hat u_{b,1}$ everywhere on $X_j$.
\item $\Flux_{u'_b}=\Flux_{\hat u_b}=\Fcal_b|_{H_1(X_j,\Z)}$; see (II$_6$). 
\end{enumerate}
Furthermore, assuming that the ball $\B'$ in (VI$_1$) is sufficiently small and $\tau>0$ is sufficiently large, conditions (II$_3$) to (II$_5$), (II$_7$), (IV$_5$), (V$_2$), (V$_3$), and (VI$_1$) guarantee the following
for every $b\in B$.
\begin{enumerate}[\rm ({i}$_{11}$)]  
\item[\rm ({i}$_4$)]   $|u'_b(p)-\hat u_b(p)|<\epsilon_{j-1}/3$ for every $p\in X_{j-1}\cup\bigcup_{a\in\Z_l^{\rm even}}D_a$.
\item[\rm ({i}$_5$)]  $u'_{b,2}(p)>j$ for every $p\in \bigcup_{a\in\Z_l^{\rm odd}}\Omega_a$.
\item[\rm ({i}$_6$)] $\min\{u'_{b,1}(p),u'_{b,2}(p)\}>j-1$ for every $p\in \bigcup_{a\in\Z_l}\gamma_a$.
\item[\rm ({i}$_7$)] $\min\{u'_{b,1}(q_a),u'_{b,2}(q_a)\}>j$ for every $a\in\Z_l$.
\item[\rm ({i}$_8$)] $\hat \varphi_{b,1}(q_a)\hat \varphi_{b,2}(q_a)\neq 0$ for every  $a\in\Z_l$.
\item[\rm ({i}$_9$)] $\hat \varphi_{b,1}(p_a)\hat \varphi_{b,2}(p_a)\neq 0$ for every  $a\in\Z_l$.
\item[\rm ({i}$_{10}$)] $\hat \varphi_{b,1}(p)\neq 0$ for every  $p\in \overline{\beta_a\setminus\Omega_a}$ and $a\in\Z_l^{\rm odd}$.
\item[\rm ({i}$_{11}$)] $\hat\varphi_{b,2}(q_a^{a-1})\hat\varphi_{b,2}(q_a^a)\neq0$ for every $a\in\Z_l^{\rm odd}$; see \eqref{eq:betaa}.
\end{enumerate}
In particular, (i$_2$) and (i$_5$)  together with (IV$_2$) and (IV$_3$) guarantee the following.
\begin{enumerate}[\rm ({i}$_{11}$)] 
\item[\rm ({i}$_{12}$)] $\max\{u'_{b,1}(p),u'_{b,2}(p)\}>j-1$ for every $p\in \bigcup_{a\in\Z_l^{\rm odd}}D_a$ and $b\in B$. 
\item[\rm ({i}$_{13}$)] $\max\{u'_{b,1}(p),u'_{b,2}(p)\}>j$ for every $p\in \bigcup_{a\in\Z_l^{\rm odd}}\beta_a$ and $b\in B$.
\end{enumerate}
This shows that the immersions $u'_b$ take suitable values on $\bigcup_{a\in\Z_l^{\rm odd}}D_a$ (see \eqref{eq:max>j} and condition (a$_j$) in the induction), but need not do so on $\bigcup_{a\in\Z_l^{\rm even}}D_a$ (see \eqref{eq:cupDa}), so we have to keep working. 

%
%
\smallskip
\noindent{\em The third deformation.}  Repeating the arguments in the second deformation but for values $a\in\Z_l^{\rm even}$ and starting with $u'$ in the place of $\hat u$, and taking into account conditions (i$_1$) to (i$_{13}$), we can obtain a continuous map $u^j:B\times X_j\to\R^3$ with the following properties for every $b\in B$, where the sets $\Omega_a$ and $\beta_a$ and the points $q_a^{a-1}$ and $q_a^a$ for $a\in\Z_l^{\rm even}$ are defined analogously to those for $a\in\Z_l^{\rm odd}$.
\begin{enumerate}[\rm ({ii}$_1$)]  
\item The map $u_b^j=(u_{b,1}^j,u_{b,2}^j,u_{b,3}^j)=u^j(b,\cdot):X_j\to\R^3$ is a $J_b$-conformal minimal immersion.
\item $u_{b,2}^j=u'_{b,2}$ everywhere on $X_j$.
\item $\Flux_{u_b^j}=\Flux_{u'_b}=\Flux_{\hat u_b}=\Fcal_b|_{H_1(X_j,\Z)}$. 
\item $|u_b^j(p)-u'_b(p)|<\epsilon_{j-1}/3$ for every $p\in X_{j-1}\cup\bigcup_{a\in\Z_l^{\rm odd}}D_a$.
\item $u_{b,1}^j(p)>j$ for every $p\in (\bigcup_{a\in\Z_l^{\rm even}}\Omega_a)\cup(\bigcup_{a\in\Z_l^{\rm odd}}\overline{\beta_a\setminus\Omega_a})$. 
\item $u_{b,2}^j(p)>j$ for every $p\in (\bigcup_{a\in\Z_l^{\rm odd}}\Omega_a)\cup(\bigcup_{a\in\Z_l^{\rm even}}\overline{\beta_a\setminus\Omega_a})$. 

\item $u_{b,1}^j(p)>j-1$ for every $p\in \bigcup_{a\in\Z_l^{\rm odd}}\overline{D_a\setminus\Omega_a}$.
\item $u_{b,2}^j(p)>j-1$ for every $p\in \bigcup_{a\in\Z_l^{\rm even}}\overline{D_a\setminus\Omega_a}$.

\item $\partial_{J_b} u^j_{b,1}(p)\neq 0$ and $\partial_{J_b} u^j_{b,2}(p)\neq 0$ for every $p\in \bigcup_{a\in\Z_l}\{q_a,q_a^{a-1},q_a^a\}$.
\end{enumerate}

It is then clear that $u^j$ satisfies inequality \eqref{eq:max>j} (see (ii$_5$), (ii$_6$), and \eqref{eq:cupbetaa}) as well as conditions (a$_j$) (see (ii$_5$) to (ii$_8$) and \eqref{eq:cupbetaa}), (b$_j$) (see (II$_2$), (i$_4$), and (ii$_4$)), and (d$_j$) (see (ii$_3$)) in the induction. (See Figures \ref{fig:Da} and \ref{fig:division}.)
%
%
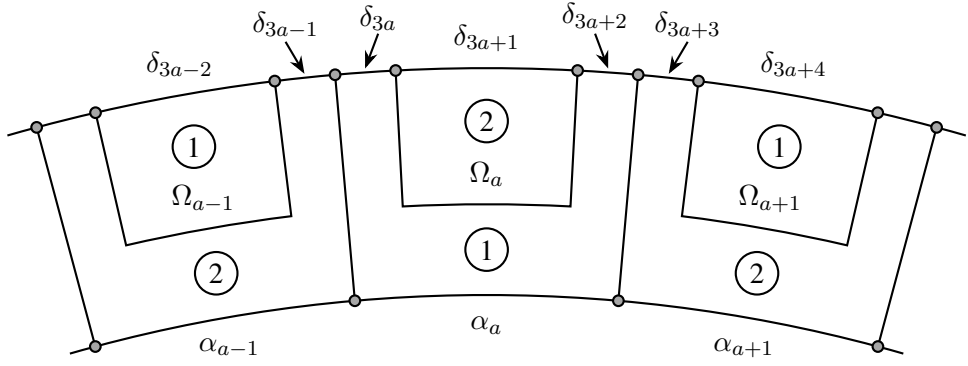
\begin{figure}[htbp]
    \centering
    \begin{tikzpicture}[>=Stealth, thick]

        \def\Rin{20}         
        \def\Rslot{21.2}     
        \def\Rout{23}        
        \def\RlabelIn{19.6}  
        \def\RlabelOut{23.4} 

        \draw (74:\Rin) arc (74:106:\Rin);
        \draw (74:\Rout) arc (74:106:\Rout);

        \foreach \a in {75, 85, 95, 105} {
            \draw (\a:\Rin) -- (\a:\Rout);
            \filldraw[fill=gray!70, draw=black] (\a:\Rin) circle (2pt);
            \filldraw[fill=gray!70, draw=black] (\a:\Rout) circle (2pt);
        }

        \draw (103:\Rout) -- (103:\Rslot) arc (103:97:\Rslot) -- (97:\Rout);
        \filldraw[fill=gray!70, draw=black] (103:\Rout) circle (2pt);
        \filldraw[fill=gray!70, draw=black] (97:\Rout) circle (2pt);

        \draw (93:\Rout) -- (93:\Rslot) arc (93:87:\Rslot) -- (87:\Rout);
        \filldraw[fill=gray!70, draw=black] (93:\Rout) circle (2pt);
        \filldraw[fill=gray!70, draw=black] (87:\Rout) circle (2pt);

        \draw (83:\Rout) -- (83:\Rslot) arc (83:77:\Rslot) -- (77:\Rout);
        \filldraw[fill=gray!70, draw=black] (83:\Rout) circle (2pt);
        \filldraw[fill=gray!70, draw=black] (77:\Rout) circle (2pt);

        \tikzset{circ/.style={circle, draw, thick, inner sep=0pt, minimum size=16pt}}

        \node[circ] at (100:22.3) {1};
        \node at (100:21.6) {$\Omega_{a-1}$};
        \node[circ] at (100:20.6) {2};
        \node at (100:\RlabelIn) {$\alpha_{a-1}$};

        \node[circ] at (90:22.3) {2};
        \node at (90:21.6) {$\Omega_a$};
        \node[circ] at (90:20.6) {1};
        \node at (90:\RlabelIn) {$\alpha_a$};

        \node[circ] at (80:22.3) {1};
        \node at (80:21.6) {$\Omega_{a+1}$};
        \node[circ] at (80:20.6) {2};
        \node at (80:\RlabelIn) {$\alpha_{a+1}$};

        \node at (100:\RlabelOut) {$\delta_{3a-2}$};
        \node at (90:\RlabelOut) {$\delta_{3a+1}$};
        \node at (80:\RlabelOut) {$\delta_{3a+4}$};

        
        \draw[<-] (96:23.05) -- (96.5:23.5) node[above, inner sep=1pt] {$\delta_{3a-1}$};
        
        \draw[<-] (94:23.05) -- (93.5:23.5) node[above, inner sep=1pt] {$\delta_{3a}$};
        
        \draw[<-] (86:23.05) -- (86.5:23.5) node[above, inner sep=1pt] {$\delta_{3a+2}$};
        
        \draw[<-] (84:23.05) -- (83.5:23.5) node[above, inner sep=1pt] {$\delta_{3a+3}$};

    \end{tikzpicture}
    \caption{The division $\Dscr^j$ of $bX_j$. Assuming that $a\in\Z_l^{\rm odd}$, the numbers $1$ and $2$ inside the circles point out the component function $u^j_{b,1}$ or $u^j_{b,2}$, respectively, 
that is suitably big in each set according to conditions (ii$_5$) to (ii$_8$).}
    \label{fig:division}
\end{figure}
%
%

By compactness of $B$, the Cauchy estimates allow us to choose a number $\epsilon_j>0$ so small that (c$_j$) holds. Finally, for each $a\in\Z_l=\{0,1,\ldots,l-1\}$ consider the arcs
\[
	\delta_{3a}=\beta_a^{a-1},\quad \delta_{3a+1}=\beta_a\cap\Omega_a,
	\quad\text{and}\quad
	\delta_{3a+2}=\beta_a^a,
\]
with the endpoints $q_{a-1}$ and $q_a^{a-1}$, $q_a^{a-1}$ and $q_a^a$, and $q_a^a$ and $q_a$, respectively; see \eqref{eq:betaa} and Figures \ref{fig:Da} and \ref{fig:division}.
It follows that $\Dscr^j=\{\delta_a: a\in\Z_{3l}\}$ is a division of $bX_j$ that, in view of properties (ii$_5$), (ii$_6$), 
and (ii$_9$), is compatible with $(u^j,j)$; see conditions ($\mathfrak D$1) to ($\mathfrak D$4) in Section \ref{ss:preparations} and recall that $l\ge 4$ is even. This closes the induction in the noncritical case.

%
%
\medskip
\noindent{\bf The critical case.} Assume that $\chi(X_j\setminus\mathring X_{j-1})=-1$, so the number of boundary components of $bX_{j-1}$ and $bX_j$ differ by one. For simplicity of exposition, we assume that $X_j\setminus\mathring X_{j-1}$ is connected, hence it is a pair of pants with boundary $bX_{j-1}\cup bX_j$. Otherwise, the other components of $X_j\setminus\mathring X_{j-1}$ are all annuli, and we argue on them as in the noncritical case. We distinguish cases.

\smallskip
\noindent{\em Case 1: $bX_{j-1}$ is connected.} So, $bX_j$ has two connected components. As in the noncritical case, write $\Dscr^{j-1}=\{\alpha_a: a\in\Z_l\}$, $l\ge 4$ even, for the given division of the closed Jordan curve $bX_{j-1}$ compatible with $(u^{j-1},j-1)$, denote by $p_a$ the only point in $\alpha_a\cap\alpha_{a+1}$, $a\in\Z_l$, and set $\Z_l^{\rm odd}$ and $\Z_l^{\rm even}$ as in \eqref{eq:Zlodd}. Thus, conditions (I$_1$) to (I$_3$) in the noncritical case are satisfied.

We may assume without loss of generality that $l\ge 8$ is a multiple of $4$. Otherwise, we choose a subarc $\alpha_0'$ of $\alpha_0$,
having $p_{l-1}$ as an endpoint and so small that 
for every $b\in B$ and $p\in \alpha_0'$ we have 
\[
	\min\{u^{j-1}_{b,1}(p),u^{j-1}_{b,2}(p)\}>j-1, 
	\quad
	\partial_{J_b} u^{j-1}_{b,1}(p)\neq 0, 
	\quad \partial_{J_b} u^{j-1}_{b,2}(p)\neq 0.
\]
Such a subarc exists by compactness of $B$ and conditions (I$_1$) 
to (I$_3$). Then, splitting $\alpha_0'$ into any family of $l$ adjacent subarcs, say $\beta_1,\ldots,\beta_l$, with $p_{l-1}\in\beta_1$, we obtain that 
$\{\beta_1,\ldots,\beta_l, \overline{\alpha_0\setminus\alpha_0'},\alpha_1,\ldots,\alpha_{l-1}\}$ is a division of $bX_{j-1}$ that is compatible with $(u^{j-1},j-1)$ and consists of $2l\ge 8$ arcs, which is a multiple of $4$ since $l$ is even.

Consider the $\Jscr$-holomorphic map $f:B\times X_{j-1}\to\bold A$ given by $f(b,\cdot)=2\partial_{J_b} u^{j-1}_b/\theta_b$ for every $b\in B$, and set $f_b=f(b,\cdot)=(f_{b,1},f_{b,2},f_{b,3}):X_{j-1}\to\boldA$.
Choose a pair of points 
\[
	q\in \alpha_0\setminus\{p_{l-1},p_0\}
	\quad\text{and}\quad 
	q'\in \alpha_{l/2}\setminus\{p_{l/2-1},p_{l/2}\}.
\]
Also, choose a smooth embedded arc $\gamma\subset \mathring X_j\setminus\mathring X_{j-1}$ with the initial point $q$, the final point $q'$, and otherwise disjoint from $bX_{j-1}$. We take $\gamma$ such that the compact set
\[
	S=X_{j-1}\cup\gamma\subset\mathring X_j
\]
is admissible (Definition \ref{def:admissible}) and a strong deformation retract of $X_j$. (See Figure \ref{fig:critical-1}.)

%
%
\begin{figure}[htbp]
    \centering
    \begin{subfigure}[c]{0.48\textwidth}
        \centering
        \vspace{-0.2cm} 
        \begin{tikzpicture}[line join=round, line cap=round, scale=0.75]
            \coordinate (centro) at (0, 2.1);
            
            \draw[dashed] (0,0) ellipse (1.5 and 0.4);
            \draw (1.5,0) arc (0:-180:1.5 and 0.4);
            
            \draw (-1.5,0) -- (-1.5,-0.7);
            \draw (1.5,0) -- (1.5,-0.7);
            
            \draw (-1.5,0) .. controls (-1.5,1) and (-3.5,1.5) .. (-3.5,3);
            \draw (1.5,0) .. controls (1.5,1) and (3.5,1.5) .. (3.5,3);
            
            \draw (-0.7,3) .. controls (-0.4,2.1) .. (centro) .. controls (0.4,2.1) .. (0.7,3);
            
            \draw (-2.1,3) ellipse (1.4 and 0.4);
            \draw (2.1,3) ellipse (1.4 and 0.4);
            
            \filldraw (0,-0.4) circle (1.5pt);
            \draw[thick] (0,-0.4) .. controls (0,0.8) and (0,1.5) .. (centro);
            \draw[dashed] (centro) .. controls (0.4,1.5) and (0.6,0.8) .. (0.6,0.36);
            \filldraw (0.6,0.36) circle (1.5pt);
            
            \node at (0,-1.2) {$X_{j-1}$};
            \node at (-2.4,2.2) {$X_{j}$};
            \node at (-0.3,1.2) {$\gamma$};
        \end{tikzpicture}
    \end{subfigure}
    \hfill 
    \begin{subfigure}[c]{0.48\textwidth}
        \centering
        \begin{tikzpicture}[line cap=round, line join=round, scale=0.85]
            \def\R{3}
            \def\r{2.5}
            \draw[thick] (0,0) ellipse ({\R} and {\r});
            
            \node[circle, fill, inner sep=1.5pt, label={above:$\alpha_0$}] (a0) at (0,\r) {};
            \node[circle, fill, inner sep=1.5pt, label={below:$\alpha_{l/2}$}] (al2) at (0,-\r) {};
            
            \node at ({140}:{\R} and {\r}) [label={[label distance=-2pt]above left:$\alpha_1$}] {};
            \node at ({220}:{\R} and {\r}) [label={[label distance=-2pt]below left:$\alpha_{l/2-1}$}] {};
            \node at ({-40}:{\R} and {\r}) [label={[label distance=-2pt]below right:$\alpha_{l/2+1}$}] {};
            \node at ({40}:{\R} and {\r})  [label={[label distance=-2pt]above right:$\alpha_{l-1}$}] {};

            \foreach \ang in {25, 65, 115, 155, 205, 245, 295, 335} {
                \node[circle, fill, inner sep=1pt] at ({\ang}:{\R} and {\r}) {};
            }

            \draw[thick] (a0) -- (al2) node[midway, right] {$\gamma$};
            
            \coordinate (a01_top) at (-0.6, 2.45);
            \coordinate (a01_bot) at (-0.6, -2.45);
            \coordinate (a02_top) at (0.6, 2.45);
            \coordinate (a02_bot) at (0.6, -2.45);

            \draw[blue, thick] (a01_top) -- (a01_bot) node[midway, left, black] {$\alpha_0^1$};
            \draw[blue, thick] (a02_top) -- (a02_bot) node[midway, right, black] {$\alpha_0^2$};

            \foreach \p in {a01_top, a01_bot, a02_top, a02_bot} {
                \node[circle, fill=blue, inner sep=1.5pt] at (\p) {};
            }
        \end{tikzpicture}
    \end{subfigure}

    \caption{Critical case 1}
    \label{fig:critical-1}
\end{figure}
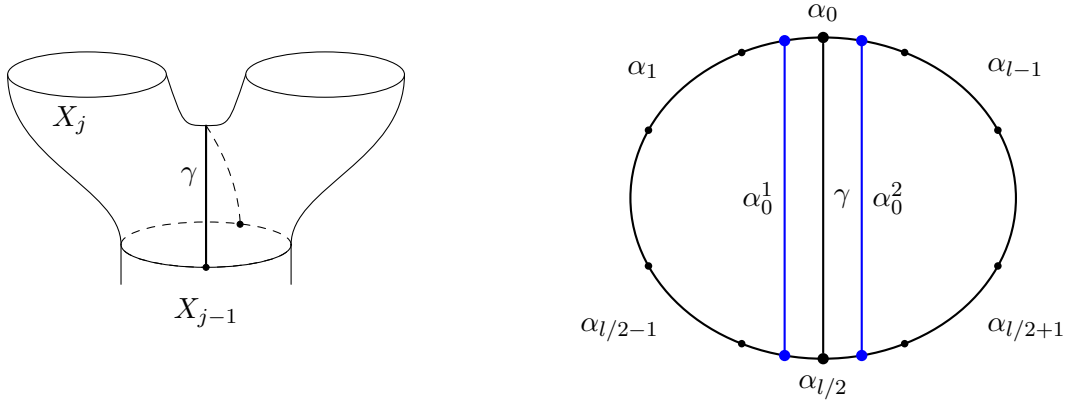
%
%
%
%

The following claim is an immediate consequence of Lemma \ref{lem:CI1};
for the second condition, take into account (I$_2$) and that $l/2$ is even.

\begin{claim}\label{cl:critical}
The pair of maps $u^{j-1}:B\times X_{j-1}\to\R^3$ and $f:B\times X_{j-1}\to\boldA$ can be extended to continuous maps $u^{j-1}:B\times S\to\R^3$ and $f:B\times S\to\boldA$ such that, setting $u^{j-1}_b=u^{j-1}(b,\cdot)=(u^{j-1}_{b,1},u^{j-1}_{b,2},u^{j-1}_{b,3})$ and $f_b=f(b,\cdot)=(f_{b,1},f_{b,2},f_{b,3})$ for $b\in B$, 
the following conditions are satisfied for every $b\in B$.
\begin{itemize}
\item The pair $(u^{j-1}_b,f_b\theta_b)$ is a generalized $J_b$-conformal minimal immersion $S\to\R^3$ of class $\Cscr^1$  
(see Definition \ref{def:admissible}).
\item $u^{j-1}_{b,2}(p)>j-1$ for every $p\in\gamma$.
\end{itemize}
\end{claim}

By the Mergelyan theorem for families of conformal minimal immersions in \cite[Theorem 8.2 and Corollary 8.6]{Forstneric2024Runge}, we can approximate $u^{j-1}$ uniformly on $B\times S$ by a continuous map $\hat u:B\times X_j\to\R^3$ satisfying the following conditions 
for every $b\in B$.
\begin{enumerate}[\rm ({A}$_1$)]
\item The map $\hat u_b=\hat u(b,\cdot)=(\hat u_{b,1},\hat u_{b,2},\hat u_{b,3}):X_j\to\R^3$ is a $J_b$-conformal minimal immersion.
\item $\hat u_{b,1}(p)>j-1$ for every $p\in \bigcup_{a\in\Z_l^{\rm odd}}\alpha_a$.
\item $\hat u_{b,2}(p)>j-1$ for every 
$p\in \gamma\cup\bigcup_{a\in\Z_l^{\rm even}} \alpha_a$.
\item $\Flux_{\hat u_b}(C)=\Fcal_b(C)$ for every closed curve $C\subset X_j$; take into account (d$_{j-1}$).
\item $\hat f_{b,1}(p_a)\hat f_{b,2}(p_a)\neq 0$ for every  $a \in\Z_l$, where 
\[
	\hat f_b=(\hat f_{b,1},\hat f_{b,2},\hat f_{b,3})
	=2\partial_{J_b}\hat u_b/\theta_b:X_j\to\boldA.
\]
\end{enumerate}

By (A$_3$) and compactness of $B$, we may choose a closed disc neighbourhood $D$ of $\gamma$ in $\mathring X_j$ such that $D\setminus\mathring X_{j-1}$ is a closed disc that intersects $bX_{j-1}$ in a pair of small closed arcs around the points $q\in\alpha_0$ and $q'\in\alpha_{l/2}$, $D\cap\alpha_a=\varnothing$ for every $a\in\Z_l\setminus\{0,l/2\}$, and
\begin{equation}\label{eq:>j-1D}
	\hat u_{b,2}(p)>j-1\quad\text{for every $p\in D$ and $b\in B$}.
\end{equation}
Moreover, we choose $D$ such that $X_{j-1}\cup D$ is a smoothly bounded compact domain that is a strong deformation retract of $X_j$. 
Set $l'=l/2$, and note by the assumption on $l$ that $l'\ge 4$ and is even. It turns out that $X_{j-1}\cup D$ has two boundary components; one of them, say $C_1$, containing the arcs $\alpha_a$ for $a=1,\ldots,l'-1$, and the other one, say $C_2$, containing the arcs $\alpha_a$ for  $a=l'+1,\ldots,l-1$. Moreover,
\[
	\alpha_0^1=\overline{C_1\setminus(\alpha_1\cup\cdots\cup\alpha_{l'-1})}
	\quad\text{and}\quad
	\alpha_0^2=\overline{C_2\setminus(\alpha_{l'+1}\cup\cdots\cup\alpha_{l-1})}
\]
are connected arcs with endpoints $p_{l'-1}$ and $p_0$, and $p_{l-1}$ and $p_{l'}$, respectively. (See Figure \ref{fig:critical-1}.) Setting $\alpha_a^1=\alpha_a$ and $\alpha_a^2=\alpha_{l'+a}$ for $a=1,\ldots,l'-1$, we have that $\Dcal_i=\{\alpha_a^i: a\in \Z_{l'}\}$ is a division of $C_i$, $i=1,2$, hence $\Dcal=\Dcal_1\cup\Dcal_2$ is a division of $b(X_{j-1}\cup D)=C_1\cup C_2$. 
Finally, (A$_2$), (A$_3$), (A$_5$), and \eqref{eq:>j-1D} guarantee that $\Dcal$ is compatible with $(\hat u,j-1)$; see conditions ($\mathfrak D$1) to ($\mathfrak D$4) in Section \ref{ss:preparations}. This and \eqref{eq:>j-1D} reduce the proof to the noncritical case.

\smallskip
\noindent{\em Case 2: $bX_{j-1}$ has two connected components.} So, $bX_j$ is connected. Write $C_1$ and $C_2$ for the components of $bX_{j-1}$ and $\Dscr^{j-1}=\Dcal_1\cup\Dcal_2$ for the given division of $bX_{j-1}$ compatible with $(u^{j-1},j-1)$, where $\Dcal_i=\{\alpha^i_a: a\in \Z_{l_i}\}$, $l_i\ge 4$ even, is a division of $C_i$ for $i=1,2$. Arguing as in the beginning of Case 1, we can assume that $l_1=l_2=l$. Denote by $p_a^i$ the only point in $\alpha_a^i\cap\alpha_{a+1}^i$, $i=1,2$, $a\in\Z_l$, and set $\Z_l^{\rm odd}$ and $\Z_l^{\rm even}$ as in \eqref{eq:Zlodd}. Conditions ({$\mathfrak D$}2) to ({$\mathfrak D$}4) then give the following properties for every $b\in B$.
\begin{enumerate}[({B}$_1$)]
\item $u^{j-1}_{b,1}(p)>j-1$ for every $p\in\alpha_a^i$, $i=1,2$, $a\in\Z_l^{\rm odd}$.
\item $u^{j-1}_{b,2}(p)>j-1$ for every $p\in\alpha_a^i$, $i=1,2$, $a\in\Z_l^{\rm even}$.
\item $\partial_{J_b} u^{j-1}_{b,1}(p_a^i)\neq 0$ and $\partial_{J_b} u^{j-1}_{b,2}(p_a^i)\neq 0$ for every $i=1,2$ and $a\in\Z_l$.
\end{enumerate}

Arguing as in Case 1, we can find a continuous map $\hat u:B\times X_j\to\R^3$, close to $u^{j-1}$ uniformly on $B\times X_{j-1}$, and a closed disc $D\subset\mathring X_j$ satisfying the following conditions (see Figure \ref{fig:critical-2}).
\begin{enumerate}[({C}$_1$)]
\item The map $\hat u_b=\hat u(b,\cdot)=(\hat u_{b,1},\hat u_{b,2},\hat u_{b,3}):X_j\to\R^3$ is a $J_b$-conformal minimal immersion for every $b\in B$.
\item $\hat u_{b,1}(p)>j-1$ for every $p\in \bigcup_{a\in\Z_l^{\rm odd}}\alpha_a^i$, $i=1,2$, and
$b\in B$.
\item $\hat u_{b,2}(p)>j-1$ for every $p\in D\cup\bigcup_{a\in\Z_l^{\rm even}} \alpha_a^i$, $i=1,2$, and  $b\in B$.
\item $\Flux_{\hat u_b}(C)=\Fcal_b(C)$ for every closed curve $C\subset X_j$ and $b\in B$; take into account (d$_{j-1}$).
\item $\hat f_{b,1}(p_a^i)\hat f_{b,2}(p_a^i)\neq 0$ for every  $a \in\Z_l$, $i=1,2$, and $b\in B$, where 
\[
	\hat f_b=(\hat f_{b,1},\hat f_{b,2},\hat f_{b,3})=2\partial_{J_b}\hat u_b/\theta_b:X_j\to\boldA.
\]
\item $D\setminus\mathring X_{j-1}$ is a closed disc that intersects $bX_{j-1}$ in a pair of small closed arcs, one in the relative interior of $\alpha_0^1$ and the other in the one of $\alpha_0^2$.
\item $X_{j-1}\cup D$ is a smoothly bounded compact domain that is a strong deformation retract of $X_j$.
\end{enumerate}

%
%
\begin{figure}[htbp]
    \centering
    \begin{subfigure}[b]{0.40\textwidth}
        \centering
        \begin{tikzpicture}[line join=round, line cap=round, scale=0.65]
          \draw[dashed] (-2,0) arc (180:0:1 and 0.25);
          \draw[thick] (-2,0) arc (180:360:1 and 0.25);
          \draw[dashed] (1,0) arc (180:0:1 and 0.25);
          \draw[thick] (1,0) arc (180:360:1 and 0.25);
          
          \draw[thick] (-2,0) -- (-2,-0.8);
          \draw[thick] (0,0) -- (0,-0.8);
          \draw[thick] (1,0) -- (1,-0.8);
          \draw[thick] (3,0) -- (3,-0.8);

          \draw[thick] (0.5,4) ellipse (0.7 and 0.2);
          \draw[thick] (-2,0) to[out=90, in=270] (-0.2,4); 
          \draw[thick] (3,0) to[out=90, in=270] (1.2,4);
          \draw[thick] (0,0) to[out=90, in=180] (0.5,1.7) to[out=0, in=90] (1,0);

          \fill (-1, -0.25) circle (2pt) coordinate (P1);
          \fill (2, -0.25) circle (2pt) coordinate (P2);
          \draw[thick] (P1) .. controls (-0.8, 2.6) and (1.8, 2.6) .. (P2);
          \node at (0.9, 2.2) {$\gamma$};

          \node at (0.5, 3.1) {$X_j$};
          \node at (-1, -1.2) {$X_{j-1}$};
          \node at (2, -1.2) {$X_{j-1}$};
        \end{tikzpicture}
    \end{subfigure}
    \hfill
    \begin{subfigure}[b]{0.55\textwidth}
        \centering
        \begin{tikzpicture}[
            dot/.style={circle, fill=black, inner sep=1.2pt},
            blue dot/.style={circle, fill=blue!80!black, inner sep=1.5pt},
            line width=0.6pt,
            scale=0.85 
        ]
            \draw (0,0) circle (2cm);
            \node[blue dot] (L_top) at (15:2cm) {};
            \node[dot]      (L_mid) at (0:2cm) {};
            \node[blue dot] (L_bot) at (-15:2cm) {};
            
            \node[dot] at (70:2cm) {}; 
            \node at (53:2.02cm) [anchor=south west] {\small $\alpha_1^1$};
            
            \node[dot] at (35:2cm) {};
            \node[dot] at (-35:2cm) {};
            
            \node[dot] at (-70:2cm) {}; 
            \node at (-53:2.02cm) [anchor=north west] {\small $\alpha_{l-1}^1$};
            
            \node[left] at (L_mid) {\small $\alpha_0^1$};

            \begin{scope}[shift={(5.2,0)}] 
                \draw (0,0) circle (2cm);
                \node[blue dot] (R_top) at (165:2cm) {};
                \node[dot]      (R_mid) at (180:2cm) {};
                \node[blue dot] (R_bot) at (-165:2cm) {};
                
                \node[dot] at (110:2cm) {}; 
                \node at (127:2.02cm) [anchor=south east] {\small $\alpha_{l-1}^2$};
                
                \node[dot] at (145:2cm) {};
                \node[dot] at (-145:2cm) {};
                
                \node[dot] at (-110:2cm) {}; 
                \node at (-127:2.02cm) [anchor=north east] {\small $\alpha_1^2$};
                
                \node[right] at (R_mid) {\small $\alpha_0^2$};
            \end{scope}

            \draw[blue!70!black, thick] (L_top) -- node[above, black] {\small $\beta'$} (R_top);
            \draw (L_mid) -- node[above] {\small $\gamma$} (R_mid);
            \draw[blue!70!black, thick] (L_bot) -- node[below, black] {\small $\beta$} (R_bot);
        \end{tikzpicture}
    \end{subfigure}

    \vspace{10pt}
    \caption{Critical case 2}
    \label{fig:critical-2}
\end{figure}
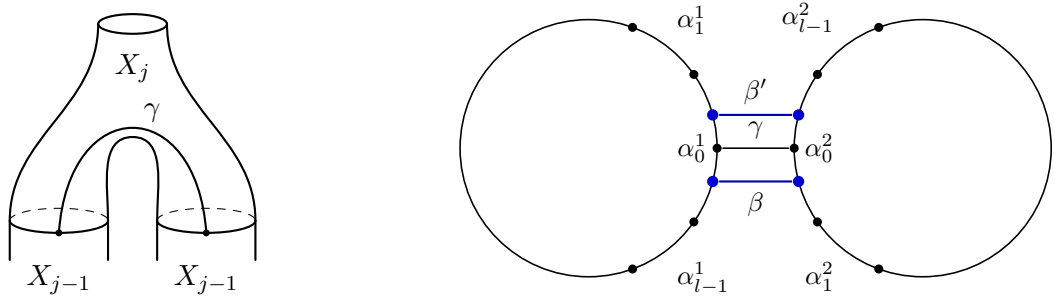

It follows that $\Gamma=b(X_{j-1}\cup D)$ is connected, it contains the arc $\alpha_a^i$ for every $a\in\Z_l\setminus\{0\}$ and $i=1,2$, and the set
\[
	\overline{\Gamma\setminus\cup_{i=1,2}(\alpha_1^i\cup\cdots\cup\alpha_{l-1}^i)}
\]
is the disjoint union of two arcs, say, $\beta$ and $\beta'$, with endpoints $p_{l-1}^1$ and $p_0^2$, and $p_{l-1}^2$ and $p_0^1$, respectively. Setting $\alpha_0=\beta$, $\alpha_a=\alpha_a^2$ for $a\in\Z_l\setminus\{0\}$, $\alpha_l=\beta'$, and $\alpha_{l+a}=\alpha_a^1$ for $a\in\Z_l\setminus\{0\}$, it turns out in view of conditions (C$_2$), (C$_3$), and (C$_5$) that $\Dcal=\{\alpha_a: a\in \Z_{2l}\}$ is a division of $b(X_{j-1}\cup D)$ that is compatible with $(\hat u,j-1)$. This and (C$_3$) reduce the proof to the noncritical case.

This closes the induction and completes the proof of the theorem.

%
%
\begin{remark}\label{rem:hp}
(a)
Our construction shows that the family of proper $J_b$-conformal
minimal immersions $u_b:X\to\R^3$ in Theorem \ref{th:CMI} can be 
chosen such that the family of their Weierstrass data 
$f_b=2\di_{J_b}u_b/\theta_b : X\to \boldA$, $b\in B$, is homotopic
to any given continuous map $B\times X\to \boldA$.
The analogous remark holds for families of proper 
null curves $X\to \C^n$ in Corollary \ref{cor:null}. 

(b) Our construction also gives a map $u$ satisfying Theorem \ref{th:CMI} such that $\max\{u_{b,1},u_{b,2}\}:X\to\R$ is a proper map for every $b\in B$. The same holds true in the context of Corollaries \ref{cor:null} and \ref{cor:C2}.
\end{remark}

%
%
%
%
\medskip 

\subsection*{Acknowledgements}
Alarc\'on is partially supported by the State Research Agency (AEI) via the grant no.\ PID2023-150727NB-I00, and the ``Maria de Maeztu'' Unit of Excellence IMAG, reference CEX2020-001105-M, funded by MICIU/AEI/10.13039/501100011033 and ERDF/EU, Spain.
Forstneri\v c is supported by the European Union
(ERC Advanced grant HPDR, 101053085) 
and grant P1-0291 from ARIS, Republic of Slovenia. 




\end{document}